\documentclass[12pt]{amsart}
\usepackage{a4wide}
\usepackage{amsthm,amsfonts,amsmath,amssymb}
\usepackage{tcolorbox}
\usepackage{mathtools}
\usepackage{booktabs}
\usepackage{xcolor}
\usepackage{dsfont}
\usepackage{tikz,tikz-cd}
\usepackage{graphicx}
\usepackage{hyperref}

\allowdisplaybreaks

\newcommand{\bF}{\mathcal{F}}
\newcommand{\eps}{\varepsilon}
\newcommand{\C}{\mathbb{C}}
\newcommand{\HH}{\mathfrak{H}}
\newcommand{\R}{\mathbb{R}}

\newcommand{\Z}{\mathbb{Z}}

\newcommand{\bN}{\mathbf{N}}
\newcommand{\bH}{\mathbf{I}}
\newcommand{\GL}{{\rm GL}}
\newcommand{\SL}{{\rm SL}}
\newcommand{\PSL}{{\rm PSL}}

\newcommand{\re}{\operatorname{Re}}
\newcommand{\im}{\operatorname{Im}}
\newcommand{\ee}{\mathbf{e}}

\newcommand{\ol}[1]{\overline{#1}}

\newcommand{\Mod}[1]{\ (\text{mod}\ #1)}

\renewcommand{\phi}{\varphi}

\newcommand*\pmat[4]{\begin{pmatrix}#1&#2\\#3&#4\end{pmatrix}}

\newcommand*\smat[4]{\begin{smallmatrix}#1&#2\\#3&#4\end{smallmatrix}}

\newtheorem*{theoremA}{Theorem A}
\newtheorem*{theorem*}{Theorem}
\newtheorem*{question}{Question}
\newtheorem{theorem}{Theorem}[section]

\newtheorem{lemma}{Lemma}[section]
\newtheorem{corollary}{Corollary}[section]
\newtheorem{conjecture}{Conjecture}[section]

\theoremstyle{definition}

\theoremstyle{plain}
\numberwithin{equation}{section} 

\author[Berghaus]{David Berghaus}
\address{Fraunhofer IAIS, Institute for Intelligent Analysis and Information  Systems, 53757 Sankt Augustin, Germany}
\email{s6ddberg@uni-bonn.de}

\author[Bondarenko]{Andriy Bondarenko}
\address{Department of Mathematical Sciences \\ Norwegian University of Science and Technology \\ NO-7491 Trondheim \\ Norway}
\email{andriybond@gmail.com}

\author[Radchenko]{Danylo Radchenko}
\address{Universit\'e de Lille, CNRS, Laboratoire Paul Painlev\'e \\ F-59655 Villeneuve d'Ascq \\ France}
\email{danradchenko@gmail.com}

\author[Seip]{Kristian Seip}
\address{Department of Mathematical Sciences \\ Norwegian University of Science and Technology \\ NO-7491 Trondheim \\ Norway}
\email{kristian.seip@ntnu.no}

\author[Sun]{Qihang Sun}
\address{Universit\'e de Lille, CNRS, Laboratoire Paul Painlev\'e \\ F-59655 Villeneuve d'Ascq \\ France}
\email{qihang.sun@univ-lille.fr}

\thanks{Bondarenko and Seip were supported in part by Grant 334466 of the Research Council of Norway. Radchenko and Sun acknowledge funding by the European Union (ERC, FourIntExP, 101078782). Part of this work was done while Bondarenko, Radchenko, and Seip did research in residence at Centre International de rencontres mathématiques (CIRM) whose support and hospitality are gratefully acknowledged.}

\title[The basis functions of Fourier interpolation]{The basis functions of Fourier Interpolation}

\date{}

\begin{document}
\begin{abstract}
	The basis functions of the Fourier interpolation formula of Radchenko and Viazovska, constructed by means of weakly holomorphic modular forms for the Hecke theta group, are entire functions of order $2$ having interesting time-frequency properties. We give precise size estimates and study the distribution of zeros of these functions. We give in particular asymptotic estimates for the location and the number of extraneous zeros on or close to the real line. This result reveals the surprising existence of Fourier nonuniqueness pairs whose apparent ``excess'' compared to the Fourier uniqueness pair of Radchenko and Viazovska may be made arbitrarily large. Our estimates also show that the basis functions fail to yield a Riesz basis in the Hilbert space used by Kulikov, Nazarov, and Sodin in their recent study of Fourier uniqueness pairs. Some numerical data are presented, suggesting additional fine scale properties.
\end{abstract}

\maketitle

\setcounter{tocdepth}{2}
\tableofcontents

\section{Introduction}
\label{sec:intro}

Our starting point is the following result of Radchenko and Viazovska~\cite{RV} about how to represent a function $f$ in terms of samples of $f$ itself and its Fourier transform
\[ \widehat{f}(\xi):=\int_{-\infty}^\infty f(x) e^{-2\pi i \xi x} dx. \]
\begin{theoremA}
	There exists a sequence of even Schwartz functions~$a_n\colon\R\to\R$
	with the property that for every even Schwartz function
	$f:\R\to\R$ we have
	\begin{equation} \label{eq:sqrt}
		f(x) = \sum_{n=0}^{\infty} f(\sqrt{n})a_n(x) + \sum_{n=0}^{\infty}\widehat{f}(\sqrt{n}) \widehat{a_n}(x) ,
	\end{equation}
	where each of the two series on the right-hand side converges absolutely
	for every real~$x$. The functions $a_n$ satisfy the following interpolatory properties: $a_n(\sqrt{m})=\delta_{n,m}$ and
	$\widehat{a_n}(\sqrt{m})=0$ when $m\ge 1$, and in addition
	\begin{equation} \label{eq:poisson} a_0(0)=\widehat{a_0}(0)=\tfrac{1}{2},  \quad a_{n^2}(0)=-\widehat{a_{n^2}}(0)=-1 , \quad a_{n}(0)=-\widehat{a_{n}}(0)=0 \quad \text{otherwise}. \end{equation}
\end{theoremA}
It was shown in \cite[Prop. 4]{RV} that  \eqref{eq:sqrt} in fact holds on the milder assumption that $f(x)$ and $\widehat{f}(x)$ are both $O((1+|x|)^{-13})$. We refer to \cite[Thm. 7.1]{BRS} for a further refinement of that result.

Following \cite{RV}, we refer to \eqref{eq:sqrt} as a Fourier interpolation formula. Thanks to recent work of  Ramos and Sousa \cite{RS2}, Kulikov, Nazarov, and Sodin \cite{KNS}, and Lysen \cite{Ly} we have an abundance of Fourier interpolation formulas of the form
\begin{equation} \label{eq:general} f(x)=\sum_{\lambda\in \Lambda} f(\lambda) g_\lambda (x)+\sum_{w\in W} \widehat{f}(w) h_w(x) \end{equation}
for suitable sequences $\Lambda$ and $W$ of real numbers and associated functions $g_{\lambda}$ and $h_{w}$.  However, in view of the symmetry $\Lambda=W$  and the interpolatory properties of the basis functions $a_n$, \eqref{eq:sqrt} stands out as a canonical case of critical time--frequency density. It contrasts for example the formulas of \cite{KNS} which are all of super-critical density, meaning that there exists an $\varepsilon>0$ such that
\[ |\{\lambda\in \Lambda : |\lambda|\le R_1 \}| + |\{w\in W: |w|\le R_2 \}| \ge (4+\varepsilon)R_1R_2 \]
holds for all sufficiently large $R_1$ and $R_2$. Here our use of the terms ``critical density'' and ``super-critical density'' is justified by a theorem of Kulikov \cite{Ku}, which asserts that under certain mild restrictions on $f$ and the sequences $\Lambda$, $W$, $g_{\lambda}$, $h_w$, a formula like \eqref{eq:general} can only hold if we have
	\[ |\{\lambda\in \Lambda : |\lambda|\le R_1 \}| + |\{w\in W: |w|\le R_2 \}| \ge 4 R_1R_2 -C
	\log^2 (4R_1 R_2) \]
for $R_1, R_2 \ge 1$ and some constant $C$.

Striking a delicate balance with the uncertainty principle, the basis functions $a_n$ appear to merit closer investigation. The functions $a_n$, henceforth called \emph{the} basis functions of Fourier interpolation, are somewhat implicitly defined, however, by means of weakly holomorphic modular forms for the Hecke theta group (see Section~\ref{sec:definitiions} for the details). Some additional effort is therefore required to unravel their basic function theoretic properties. The objective of the present paper is to carry out that work in some detail, beyond what was done in \cite{RV,BRS}, to get a precise grip on the nature of the basis functions of Fourier interpolation.

Following \cite{RV}, we decompose $a_n$ as a sum of eigenfunctions of the Fourier transform, i.e., we write $a_n=(b_n^++b_n^-)/2$, where $\widehat{b_n^+}=b_{n}^+$ and $\widehat{b_n^-}=-b_n^-$. Since the properties of $b_n^{\pm}(x)$ are similar, we will concentrate only on the ``$+$'' case. As was proved in~\cite{RV}, the functions $b_{n}^{+}(x)$ are Fourier invariant Schwartz functions uniquely characterized by the property $b_{n}(\sqrt{m})=\delta_{n,m}$ for all $m\ge0$. It is convenient to change variables and consider
	\[f_n(x) := b_{n}^{+}(\sqrt{x}),\quad x\ge0\,.\]
The properties of these functions are the main subject of interest in this work. A definition of $f_n(x)$ as an integral transform of a certain modular form is given in Section~\ref{sec:definitiions}.

From \cite[Sect. 7]{BRS} we know that $f_n(x)\ll n^{1/4}(1+\log^2 n)$ holds uniformly in $x$. We are unable to improve that estimate in the uniform aspect, but the following bounds clarify quite precisely how the size of $|f_n(x)|$ depends on $x$ for $x>0$.

\begin{theorem}[Asymptotics for small real values] \label{thm:positiverealsmall}
	We have
	\begin{equation}
		\label{eq:outsidepw}
		|f_n(x)| \ll (1+\log^2 n) n^{1/4} (1+x)^{-1/4},\quad 0\le x \le n/8.
	\end{equation}
\end{theorem}
For larger values of $x$ the picture is much more precise.
\begin{theorem}[Asymptotics for large real values] \label{thm:positivereallarge}
	We have
    	\begin{equation}
		\label{eq:pw}
		\frac{f_n(x)}{\sin\pi(x-n)} =
        \begin{cases} \displaystyle
        O\Big(\frac{1}{\sqrt{n}}\Big(\frac{n-x}{\sqrt{n}}\Big)^{0.216}\Big), \qquad n/8 \le x < n-\sqrt{n},\\[15pt]
       \displaystyle
       \frac{1}{2\pi \sqrt{n} (\sqrt{x}-\sqrt{n})} + O(n^{-1/2}), \qquad n-\sqrt{n} \le x \le n+\sqrt{n},\\[15pt]
        \displaystyle
        \frac{e^{-\pi\sqrt{3}(\sqrt{x}-\sqrt{n})}}{\sqrt{n}} + O\Big(\frac{e^{-\pi\sqrt{11}(\sqrt{x}-\sqrt{n})}}{\sqrt{n}}\Big), \quad n+\sqrt{n} < x.
        \end{cases}
        \end{equation}
\end{theorem}
The proof of Theorem~\ref{thm:positivereallarge} is based on the approximation (see Lemma~\ref{prop:fnapprox})
    \begin{equation*}% \label{eq:fnapproximation}
    \frac{f_n(z)}{\sin\pi(z-n)} = \frac{\Phi(\sqrt{z}-\sqrt{n})}{\sqrt{n}} + O_{\eps}\big(e^{\pi\sqrt{3}(\sqrt{n}/3-\re\sqrt{z})}\big)\,,
    \end{equation*}
valid for $\re \sqrt{z}\ge \sqrt{n}(\tfrac{1}{3}+\eps)$ (for any $\eps>0$), where the special function $\Phi(w)$ is meromorphic in $\C$, and defined for $\re w>0$ by the series
    \[\Phi(w) \coloneqq \sum_{n=0}^{\infty}\frac{e^{-2\pi w\sqrt{2(n+3/8)}}}{\sqrt{2(n+3/8)}}.\]
(The analytic continuation of $\Phi$ to $\C^{\times}$ is given in~\eqref{eq:phitaylor}.)
The three regimes in~\eqref{eq:pw} correspond to the behavior of $\Phi(w)$ for large positive $w$, for $w$ close to $0$, and for large negative values of $w$, respectively. 
The more difficult part of~\eqref{eq:pw} is the range $n/8\le x \le n-\sqrt{n}$, corresponding to $\Phi(w)$ for large negative $w$; it turns out that for $w<0$ the function $\Phi(w)$ behaves like the exponential sum
    \[\sum_{n=1}^{\infty}\frac{2\cos\big(\tfrac{3n-1}{4}\pi-\tfrac{w^2}{n}\pi\big)}{\sqrt{n}}.\]
Estimation of this sum gives the exponent $0.216$ in~\eqref{eq:pw}; it may be reduced to an arbitrary $\varepsilon>0$ according to the conjecture of exponent pairs. The effect of this exponential sum is visible in the plot of $f_{500}$ shown in Figure~\ref{fig:f500a}, which may also be compared with the estimates of Theorem~\ref{thm:positivereallarge}.

\begin{figure}[ht]
    \includegraphics[width=16cm,height=6.5cm]{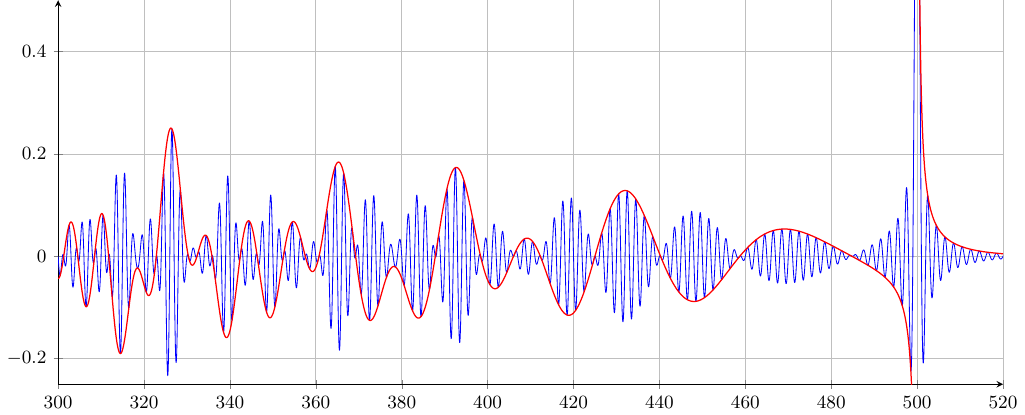}
    \caption{Plots of $f_{500}(x)$ (blue) and 
    $\frac{\Phi(\sqrt{x}-\sqrt{500})}{\sqrt{500}}$ (red).}
    \label{fig:f500a}
\end{figure}

Computing the second moment of the above exponential sum (see Theorem~\ref{thm:moment} below), we find that on average $f_n$ is indeed quite small on $[n/8,n]$, namely,
    \begin{equation} \label{eq:fnsecmom} 
    \int_{n/8}^n |f_n(x)|^2 dx \asymp \log n .\end{equation}
On the other hand, as we will see in Corollary~\ref{cor:RB}, the fact that the integrals in \eqref{eq:fnsecmom} are unbounded when $n\to \infty$, rules out the possibility that the series in  \eqref{eq:sqrt} can converge in the same robust way as the Fourier interpolation formulas of \cite{KNS}, i.e., in the sense of frames in the Hilbert space $\mathcal{H}$ which consists of functions $f$ satisfying 
\[  \| f\|_{\mathcal{H}}^2\coloneqq \int_{-\infty}^\infty \big(|f(x)|^2 +|\widehat{f}(x)|^2\big)(1+x^2) dx 
%+\int_{-\infty}^\infty |\widehat{f}(x)|^2 (1+x^2) dx
<  \infty.\]

Theorem~\ref{thm:positiverealsmall} and Theorem~\ref{thm:positivereallarge} together yield the $L^2$ bound
	\begin{equation} \label{eq:L2weak} \int_0^\infty |f_n(x)|^2 dx \ll n \log^4 n . \end{equation}
Averaging with respect to $n$, we get a result that indicates that there is room for a dramatic improvement of Theorem~\ref{thm:positiverealsmall}.

\begin{theorem}[Average of $L^2$ integrals with respect to $n$]\label{thm:sqrsum}
	We have
	\begin{equation} \label{eq:sqrsum} \xi \log \xi \ll  \int_0^\infty \sum_{n\le \xi} |f_n(x)|^2
		dx \ll \xi \log^2 \xi  \end{equation}
	for $\xi \ge 2$.
	%for every $\varepsilon >0$.
\end{theorem}
Theorem~\ref{thm:sqrsum} shows that the logarithmic power on the right-hand side of \eqref{eq:L2weak} can be reduced to $\log^2 n$, but it suggests the much stronger assertion that the integral in \eqref{eq:L2weak} is $O(\log^2 n)$. The lower bound in \eqref{eq:sqrsum} is immediate from \eqref{eq:fnsecmom}, so the main point of interest of Theorem~\ref{thm:sqrsum} is the upper bound, which is obtained by a refinement of the techniques from \cite{BRS} in combination with a result about $L^1$ Carleson measures for $H^2$ of the upper half-plane.
%	\[  \int_0^\infty \Big(\sum_{n\le \xi} |f_n(x)|^2 \Big) dx \gg  \xi \]
%is immediate from \eqref{eq:pw}, so the difficult part of the proof of Theorem~\ref{thm:sqrsum} is to establish the upper bound. 
In Section~\ref{sec:numerics}, we give numerical evidence that the correct rate of growth in~\eqref{eq:sqrsum} should be $\xi\log\xi$.

We will refer to the zeros of $f_n(z)/\sin \pi (z-n)$ as the extraneous zeros of $f_n$. It turns out that there is an abundance of such zeros on or close to the positive real line when $n$ increases. We will obtain fairly precise estimates on the number and location of these zeros, and we will show that a good part of them actually lie on the real line. Our study of the extraneous zeros of $f_n$ will require the following asymptotic estimates.

\begin{theorem}[Asymptotics in $\mathbb C$ and on the negative real line] \label{thm:global}
	We have
	\begin{equation}
		\label{eq:cartw}
		f_n(x+iy) \ll  n^{\frac14} (1+\log^2 n) e^{\pi |y|} e^{2\pi \sqrt{n} \operatorname{Re} \sqrt{-z}}
	\end{equation}
	uniformly in $z=x+iy$ and $n$. Moreover,
	\begin{align} \label{eq:negint}
		f_n(-m) & = \frac{e^{2\pi \sqrt{nm}}}{2\sqrt{n}}
        \Big(1+O\big(m^2n^2e^{-\pi\sqrt{mn}}\big)\Big)\,, \quad m, n\in \Z_{>0} ; \\
		\label{eq:tomininf}	f_n(-x) & \sim \frac{e^{2\pi \sqrt{nx}}}{2\sqrt{n}}\,,
		\quad x \to +\infty .
	\end{align}
\end{theorem}
The global bound \eqref{eq:cartw} shows that $f_n$ belongs to the Cartwright class, which means that~$f_n$ is of exponential type $\pi$ and that
	\[ \int_{-\infty}^{\infty} \frac{\log^+ |f_n(x)|}{1+x^2} dx < \infty. \]
Quite accurate information about the global distribution of zeros of $f_n$ is therefore readily available. To state a precise result, we denote by $z_k$ the zeros of $f_n$, counting multiplicities in the usual way. Let $n_\pm(r,\alpha)$ be the number of zeros $z_k$ in the sectors $\{z: |z|<r, |\arg z|\le \alpha \}$ and $\{z: |z|<r, |\arg z-\pi|\le \alpha \}$, respectively. Then by a classical theorem in the theory of entire functions (see \cite[Ch. 17]{L}), we have the following.
\begin{itemize}
	\item[(a)] The Blaschke condition holds:
	\[ \sum_{z_k\neq 0} |\im (1/a_k)|  < \infty.\]
	\item[(b)] We have
	\[ \lim_{r\to \infty} \frac{n_+(r,\alpha)}{r} = \lim_{r\to \infty} \frac{n_-(r,\alpha)}{r}=1\] for every $\alpha$ in $(0,\pi)$.
	\item[(c)] The limit \[ \lim_{R\to \infty} \sum_{0<|a_k|\le R} \re (1/z_k) \]  exists.
\end{itemize}
We see from (b) that the density of zeros in any narrow sector $|\arg z - \pi | <\varepsilon$ is $1$, but, in view of \eqref{eq:tomininf}, only finitely many of them can lie on the real line. This is in contrast to what we have in the ``opposite'' sector $|\arg z | < \varepsilon$, where indeed most of the zeros lie on the real line. To quantify more precisely the amount of extraneous zeros on or close to the positive real half-line, it is convenient to introduce the set 
  \[  \Delta(t,n)\coloneqq \{z: \  0<\operatorname{Re}z\le n, \ 3\le \operatorname{Re} (\sqrt{z}-\sqrt{n})^2 \le t, |\operatorname{Im} (\sqrt{z}-\sqrt{n})^2|\le \log n \}.\]
Our results about extraneous zeros can now be summarized as follows.  
\begin{theorem} \label{thm:zeros} The function $f_n(z)/\sin \pi z$ has for sufficiently large $n$
	\begin{itemize}
		\item[(i)] $x+O(1)$ zeros in $\Delta(x,n)$ for $3< x \le n/2 $; 
		\item[(ii)] $\gg \sqrt{x}/\log^2 n$ real zeros in $\Delta(x,n)$ for $\log^2 n <x< n/2$;
		\item[(iii)] $\gg x^{0.784}$ nonreal zeros in $\Delta(x,n)$ for $\log n<x<n/2$;
        \item[(iv)]
		$\asymp \sqrt{rn}$ zeros in the rectangle $0\le \re z \le r$, $|\im z|\le r$
		when $10\le r \le n/8$.
	\end{itemize}
\end{theorem}

Theorem~\ref{thm:zeros} has a curious implication concerning Fourier uniqueness and nonuniqueness pairs. To make this point precise, we begin by recalling from \cite{KNS} that $(\Lambda,W)$, with $\Lambda$ and $W$ two sequences of real numbers, is said to be a Fourier uniqueness pair for a class of functions $H$ if the vanishing of a function $f$ in $H$ on $\Lambda$ and that of its Fourier transform $\widehat{f}$ on $W$ imply that $f\equiv 0$. Conversely, the pair $(\Lambda,W)$ is said to be a Fourier nonuniqueness pair for $H$ if there exists a nontrivial function $f$ in $H$ with $f$ and $\widehat{f}$ vanishing respectively on $\Lambda$ and $W$. We may in this definition allow a point to appear a multiple number of times in $\Lambda$ or $W$ and interpret in the usual way the vanishing of $f$ or $\widehat{f}$ at such a point in terms of multiplicities.

Now let $\Xi_n$ be the set of extraneous positive real zeros of $b_n$. According to Theorem~\ref{thm:zeros}, the cardinality of $\Xi_n$ is $\gg \sqrt{n}/\log^2n$. We now set  $\Lambda_n\coloneqq  \{\sqrt{k}: k\ge 0, k\neq n\} \bigcup \Xi_n$ and observe that $(\Lambda_n,\Lambda_n)$ is a Fourier nonuniqueness pair for even Schwartz functions. The striking point is that if we trade two points from the sequence $\Lambda\coloneqq \{ \sqrt{k}: k\ge 0\}$ (for which $(\Lambda,\Lambda)$ is a Fourier uniqueness pair for even Schwartz functions) for $\gg \sqrt{n}/\log^2 n$ points, then we end up with a Fourier nonuniqueness pair.
 
\subsection*{Outline of the paper} This paper contains 7 additional sections. We present the basic definitions in Section~\ref{sec:definitiions}, including a very brief discussion of Kloosterman sums. In Section~\ref{sec:rademacher}, we deduce Rademacher-type formulas for the coefficients of weakly holomorphic modular forms $g_n$ that appear in the definition of $f_n$, which paves the way for the definition of the special function $\Phi$ in Section~\ref{sec:Phi}. A key point of Section~\ref{sec:Phi} is that $\Phi$ satisfies a functional equation which allows us to estimate it in terms of a certain exponential sum. We then prove Theorem~\ref{thm:positivereallarge}, establish the second moment of $\Phi$ (which in turn yields \eqref{eq:fnsecmom}), and deduce also items (i)--(iii) of Theorem~\ref{thm:zeros}. In Section~\ref{sec:positivesmall}, we develop the techniques from \cite{BRS} to prove Theorem~\ref{thm:positiverealsmall} and the upper bound in Theorem~\ref{thm:sqrsum}. In Section~\ref{sec:globalzeros}, we apply a version of the circle method to establish Theorem~\ref{thm:global}. We then  combine that theorem suitably with Jensen's formula to deduce the final item (iv) of Theorem~\ref{thm:zeros}. The fact that the basis functions do not give a frame (or more precisely a Riesz basis) for the even functions in~$\mathcal{H}$ is established in Section~\ref{sec:RB}. Finally, in Section~\ref{sec:numerics}, we present some numerical work and data to obtain more insight into the behavior of $f_n$. Based in part on this information, we formulate a number of questions and conjectures which we hope may inspire further investigations.

\section{Definitions}
\label{sec:definitiions}
In this section we briefly recall the construction of $f_n(x)$ (or $b_n^{+}(x)$) from~\cite{RV} via contour integrals of weakly holomorphic modular forms for the theta group $\Gamma_{\theta}$. We will also recall some basic definitions about Kloosterman sums that will be used in Section~\ref{sec:rademacher}.

\subsection{Weakly holomorphic modular forms \texorpdfstring{$g_n(\tau)$}{g\_m(tau)}}
Recall that
	\[\Gamma_{\theta}=\big\{\gamma \in \PSL_2(\Z)\colon \gamma\equiv (\smat 1001) \mbox{ or }  (\smat 0110) \Mod{2} \big\}\]
is generated by $S=(\smat{0}{-1}{1}{0})$ and $T^2=(\smat 1201)$ (with $T=(\smat 1101)$), which are subject to the one relation $S^2=1$. We will use the following fundamental domain for $\Gamma_{\theta}$: 
    \[\mathcal{F}:=\{\tau\in\HH\colon |\re\tau|<1, |\tau|>1\},\]
where $\HH$ denotes the upper half-plane in the complex plane (see Figure~\ref{fig:jell}).

Let $j(\gamma,\tau)=c\tau+d$ be the automorphic factor for $\gamma=(\smat abcd)$ in $\Gamma_\theta$, and define the Jacobi theta function as
\begin{equation}\label{eq:define_theta}
    \theta(\tau):=\sum_{n\in \Z} e^{\pi i n^2 \tau}.
\end{equation}
It satisfies the transformation law
    \begin{equation}\label{eq:theta_feq}
    \theta(\gamma \tau)=\nu_\theta(\gamma) j(\gamma, \tau)^{1/2}\theta(\tau),\quad \gamma\in \Gamma_\theta,
    \end{equation}
where the argument of $j(\gamma,\tau)$ is fixed to lie in $(-\pi,\pi]$ and $\nu_\theta$ is a weight $1/2$ multiplier system given by
\begin{equation}
    \nu_\theta(\gamma):=\left\{\begin{array}{ll}
        \eps_d^{-1}(\tfrac{2c}d), &\ c\equiv 0\Mod 2 ,\\
        \ee(-\tfrac 18)\eps_c (\tfrac{2d}c), &\ c\equiv 1\Mod 2, 
    \end{array}\right. \quad 
    \eps_d:=\left\{\begin{array}{ll}
         1, &\ d\equiv 1\Mod 4, \\
         i, &\ d\equiv 3\Mod 4, \\
         0, &\ \text{otherwise},
    \end{array}\right. 
\end{equation}
for $c>0$, and $\nu_\theta(-\gamma):=i\nu_\theta(\gamma)$ for $c<0$. Here $(\tfrac{a}{n})$ is the Kronecker symbol (see \cite[p.~32]{Mu} and \cite[\S2]{RadchenkoSun} for details) and
    \[\ee(\tau):=e^{2\pi i \tau}\,, \qquad \tau \in \C\,.\]

By $g_n(\tau)$, $n\ge0$, we will denote the unique weakly holomorphic modular form of weight $3/2$ for the group $\Gamma_{\theta}$, with multiplier system $\nu_{\theta}^{3}$, such that it vanishes at the cusp $1$ and satisfies
	\[g_{n}(\tau)=e^{-\pi i n\tau}+O(e^{\pi i \tau})\]
as $\tau\to i\infty$ (see~\cite[\S3]{RV}). The $0$th function $g_0(\tau)$ is just $\theta^3(\tau)$ (with $\theta$ as in~\eqref{eq:define_theta}), and for all $n\ge0$, we have $g_n(\tau)=\theta^3(\tau)Q_n(J(\tau))$, with $Q_n$ a polynomial of degree $n$, and 
    \[J(\tau):=\frac{16}{\lambda(\tau)(1-\lambda(\tau))} 
    %=\Big(\frac{\theta(\tau)}{\eta(\tau)}\Big)^{12}
    = q^{-1/2}+24+276q^{1/2}+2048q+\cdots,\]
where $\lambda(\tau)$ is the modular lambda function and as usual
    \[q^a:=e^{2a\pi i \tau}.\] 
($J$ is a Hauptmodul for the group~$\Gamma_{\theta}$, see Figure~\ref{fig:jell}; we follow the normalization in~\cite[\S3.2.2]{BRS}.) There is no simple formula for the polynomials $Q_n$, but the generating series for $g_n(z)$ can be given in closed form:
	\begin{equation} \label{mkernel}
	K(\tau,z):=\sum_{n\ge0}g_n(z)e^{\pi i n\tau} = \theta^3(z)\frac{\theta(\tau)(1-2\lambda(\tau))J(\tau)}{J(\tau)-J(z)}\,,
	\end{equation}
valid for $\im\tau$ sufficiently large. From this we see that
	\begin{equation*}
	%\label{eq:gn}
	g_n(z)=\frac{1}{2} \int_{\tau_0-1}^{\tau_0+1} K(\tau, z) e^{-\pi i n \tau} d\tau
	\end{equation*}
for any $\tau_0$ such that $\tau\mapsto K(\tau,z)$ is analytic in $\im\tau>\im\tau_0$.

\begin{figure}[ht]
    \centering
    \begin{tikzpicture}
    \definecolor{cv0}{rgb}{0.95,0.95,0.95}
    \definecolor{cv1}{rgb}{0.90,0.90,0.90}
    \clip(-9,-0.5) rectangle (7,3.5);
    \begin{scope}[scale=0.5,xshift=-9cm]
    \draw[lightgray] (-5,0) -- (5,0);
    \fill[color=cv0] (-3,0) arc (180:90:3) -- (0,6.5) -- (-3,6.5);
    \fill[color=cv1] (0,3) arc (90:0:3) -- (3,6.5) -- (0,6.5);
    \draw[green] (-3,0)  --  (-3,6.5);
    \draw[green] (3,0)  --  (3,6.5);
    \draw[red] (-3,0) arc  (180:0:3);
    \draw[lightgray,dashed] (0,3)  --  (0,6.5);
    \draw (0.5,4.5) node[above]{$\bF$};

    \fill[black] (0,3) circle (0.07) node[below] {$i$};
    \fill[black] (-3,0) circle (0.07) node[below] {$-1$};
    \fill[black] (3,0) circle (0.07) node[below] {$1$};

    \draw (-4,4.5) node[above]{$\HH$};
    \end{scope}

    \draw [->] (-1.5,1) -- (0,1);
    \draw (-0.7,1) node[above] {$J$};

    \begin{scope}[scale=0.5,xshift=7cm]
    \fill[color=cv0] (-4,2.5) -- (4,2.5) -- (4,6.5) -- (-4,6.5);
    \fill[color=cv1] (-4,2.5) -- (4,2.5) -- (4,-4.5) -- (-4,-4.5);

    \draw[green] (-4,2.5)  --  (-1,2.5);
    \draw[red] (-1,2.5)  --  (1,2.5);
    \draw[lightgray,dashed] (1,2.5)  --  (4,2.5);
    \fill[black] (-1,2.5) circle (0.07) node[above] {$0$};
    \fill[black] (1,2.5) circle (0.07) node[above] {$64$};

    \draw[blue,dashed] (-1,2.5)  --  (-1,1.5);
    \draw[blue,dashed] (-1,1.5)  --  (1,1.5);
    \draw[blue,dashed] (1,1.5)  --  (1,2.5);
    \draw (0,0.6) node[above]{{\small $\ell$}};
    \draw (-3,4.5) node[above]{$\mathbb{C}$};
    \end{scope}
    \end{tikzpicture}
    \caption{Fundamental domain $\bF$, $J(z)$ as a conformal map, and a deformed integration path $\ell$.}
    \label{fig:jell}
\end{figure}
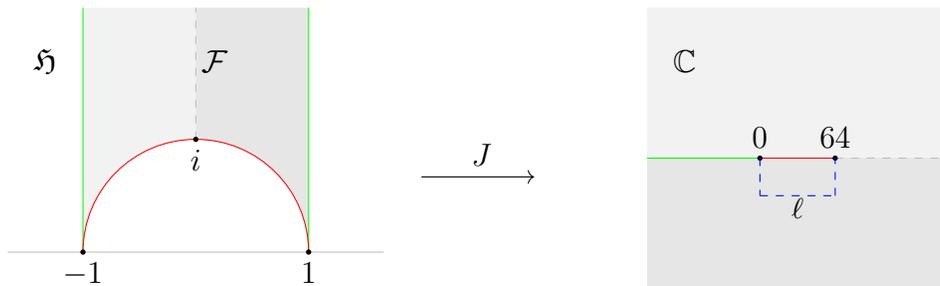

Let us also recall \cite[Eq. (3,4), (3.5), (3.6)]{BRS}: 
the Fourier expansions at 1
    \begin{equation} \label{eq:asymptexp}
    \begin{split}
    \tfrac{1}{2}(\tfrac{\tau}{i})^{-1/2}\theta(1-\tfrac{1}{\tau})
    &= q^{1/8}+q^{9/8}+q^{25/8}+\dots,\\
    -2^{-12}J(1-\tfrac{1}{\tau})   &= q+24q^2+300q^3+\dots,\\
    8\big(1-2\lambda(1-\tfrac{1}{\tau})\big) &=
    q^{-1/2}+20q^{1/2}-62q^{3/2}+\dots,
   	\end{split}
    \end{equation}
the differential identity
    \begin{equation*} \label{eq:jderiv}
    J'(\tau) = {-\pi i}\,\theta^4(\tau)J(\tau)(1-2\lambda(\tau)),
    \end{equation*}
and its corollary
	\begin{equation} \label{eq:theta4dz}
	\theta^4(\tau)d\tau = \pi^{-1}w^{-1/2}(64-w)^{-1/2}dw\,,\qquad w=J(\tau).
	\end{equation}

\subsection{Definition of \texorpdfstring{$f_n(x)$}{f\_n(x)}}
Following~\cite[Eq. (22)]{RV}, we define the function $b_n^{+}$ as
	\[b_n^{+}(x) = \frac{1}{2}\int_{-1}^{1}g_n(z)e^{\pi i z x^2} dz\,,\]
where the integral is taken over a semicircle in the upper half-plane.
In~\cite{RV} it is shown that $b_n^{+}$ is the unique even Schwartz function such that $\widehat{b_n^{+}}=b_n^{+}$, and $b_{n}^{+}(\sqrt{m})=\delta_{m,n}$, $m,n\in\Z_{\ge0}$.
We set $f_n(x)=b_{n}^{+}(\sqrt{x})$, so that
	\begin{equation} \label{eq:fndef1}
	f_n(x) = \frac{1}{2}\int_{-1}^{1}g_n(z)e^{\pi i z x} dz\,,\qquad x\ge0\,.
	\end{equation}
From~\eqref{eq:fndef1} we see that $f_n(w)$ is an entire function of exponential type $\pi$. Changing the contour of integration (see~\cite[p.~79]{RV}), we obtain the following Laplace transform representation of $f_n$:
	\begin{equation} \label{eq:fndef2}
	f_n(x) = \sin(\pi x) \int_{0}^{\infty}g_n(1+it)e^{-\pi x t} dt\,,\qquad x>n\,.
	\end{equation}
We will make use of it in Section~\ref{sec:approx-cn}.

For every $x\ge0$, the sequence $\{f_n(x)\}_{n\ge0}$ is uniquely characterized among all polynomially growing sequences of complex numbers by the functional equation
    \begin{equation} \label{eq:Ffeq}
    F(\tau,x)+(\tau/i)^{-1/2}F(-1/\tau,x) = e^{\pi i x\tau}+(\tau/i)^{-1/2}e^{-\pi i x\tau^{-1}},
    \end{equation}
where
	\begin{equation} \label{eq:gfseries}
    F(\tau,x) := \sum_{n=0}^{\infty} f_n(x) e^{\pi i n \tau}
    \end{equation}
is the generating function of $f_n(x)$.
From the generating function for $g_n(\tau)$ we also get the integral representation
	\begin{equation} \label{eq:fngf}
	F(\tau,z) = \frac{1}{2} \int_{-1}^1 K(\tau,w) e^{i \pi w z} dw\,,
	\end{equation}
which holds for $\tau$ in the fundamental domain $\mathcal{F}$.

\subsection{Kloosterman sums}
We mainly follow the notation for generalized Kloosterman sums from \cite[Ch.~4]{IwaniecBookTopicsAF}. Recall that $\Gamma_{\theta}$ has two cusps: $\infty$ and $1$. The stabilizer of the cusp $\infty$ is $\Gamma_\infty=\{\pm (\smat 1{2t}01): t\in \Z\}$ and the width of $\infty$ is $h=2$. The stabilizer of the cusp $1$ is $\Gamma_1=\{\pm (\smat {1-t}{t}{-t}{1+t}): t\in \Z\}$, and we take the scaling matrix $\sigma_1=\big(\smat{1/\sqrt 2}{\ -\sqrt 2}{1/\sqrt 2}0\big)$  in $\GL_2(\R)$ satisfying $\sigma_1\infty=1$ and $\sigma_1^{-1}\Gamma_1\sigma_1=\Gamma_\infty$. 

The cusp parameters are $\alpha_\infty=0$ and $\alpha_1=3/8$, since
\[\nu_\theta^3(T^2)=1\text{ \ and \ }\nu_{\theta}^3(\sigma_1(\smat 1201) \sigma_1^{-1})=\nu_\theta^3(\smat 01{-1}2)=\big( i\,\nu_\theta(\smat 0{-1}1{-2})\big)^3=\ee(\tfrac 38).\]
We have the double coset decomposition
\begin{align}\label{eq:doubleCosetDecomp_sigma1}
		\begin{split}
			&\Gamma_\theta\sigma_1 =\left\{\pmat{\frac{a+b}{\sqrt 2}}{-\sqrt 2\,a}{\frac{c+d}{\sqrt 2}}{-\sqrt 2\,c}:(\smat abcd )\in \Gamma_\theta\right\}\\
			&=\bigcup_{\text{odd }C>0}\Gamma_\infty \Big\{\pmat{A/\sqrt 2}{\sqrt2\,B}{C/\sqrt 2}{\sqrt 2\,D}: \begin{array}{l}
				A\Mod{2C}^*,\ D\Mod C^*\\
				AD\equiv 1\Mod C,\ A\geq 0,\ D>0,
			\end{array} 
			\Big\} \Gamma_\infty.
		\end{split}
	\end{align}
Now we define our Kloosterman sums:
\begin{equation}
    S(m,n,c)\vcentcolon =S_{\infty\infty}(m,n,c,\nu_\theta^3)=\sum_{\gamma=(\smat a*cd)\in \Gamma_\infty\setminus \Gamma_\theta/\Gamma_\infty} \nu_\theta(\gamma)^{-3} \ee\left(\frac{ma+nd}{2c}\right).
\end{equation}
Explicitly, we have
\begin{equation}
    S(m,n,c)=\left\{\begin{array}{ll}
        \displaystyle  \sum_{\substack{a,d\Mod {2c}\\ad\equiv 1\Mod {2c}}} \eps_d^{-1}(\tfrac{2c}d) \ee\left(\frac{ma+nd}{2c}\right),&\ c\equiv 0\Mod 2,  \\
        \displaystyle \sum_{\substack{\text{even }a,d\Mod{2c}\\ad\equiv 1\Mod c}}\hspace{-10px}\ee(\tfrac 38)\eps_c(\tfrac{2d}c)\ee\left(\frac{ma+nd}{2c}\right),&\ c\equiv 1\Mod 2.  
    \end{array}\right. 
\end{equation}
For odd $C>0$, $\gamma=\pmat{A/\sqrt 2}{\sqrt2\,B}{C/\sqrt 2}{\sqrt 2\,D}$ in \eqref{eq:doubleCosetDecomp_sigma1}, we define
\begin{align}
    \begin{split}
        \tilde{S}(m,n,C)&\vcentcolon=S_{\infty 1}\left(m,n,\frac{C}{\sqrt 2},\nu_\theta^3\right)  =  \hspace{-10px}\sum_{\gamma\in \Gamma_\infty\setminus \Gamma_\theta\sigma_1/\Gamma_\infty}   \hspace{-10px}(\nu_\theta)_{\infty 1}^{-3}(\gamma) \,\ee\left(\frac{m\frac{A}{\sqrt 2}+n_+\sqrt 2\,D}{2C/\sqrt 2}\right)\\
        &=\sum_{\substack{D\Mod C^*\\\text{odd }A\Mod{2C}^*\\AD\equiv 1\Mod C}} \nu_\theta\pmat**{-D}{C+D}^{-3}\,\ee\left(\frac{mA+2n_+D}{2C}\right).
        \end{split}
\end{align}
In the above formula, $n_{+}=n+3/8$. The Kloosterman sums $S$ and $\tilde{S}$ satisfy some nontrivial relations. By \cite[Prop.~2.2]{RadchenkoSun}, for odd $c>0$,
\begin{equation*}%\label{eq:RelationKLsums}
    S(m,4n,2c)=\left\{\begin{array}{ll}
        -\sqrt2\, S(m,n,c), &\ m\equiv 0,3\Mod 4,  \\
         \sqrt2\, S(m,n,c), &\ m\equiv 1,2\Mod 4. 
    \end{array}\right.
\end{equation*}
Similarly, for odd $c>0$, 
\begin{equation}\label{eq:RelationKLsumsTilde}
    \ee(-\tfrac 38)S(m,8n+3,2c)=\left\{\begin{array}{ll}
        -\sqrt2\, \tilde{S}(m,n,c), &\ m\equiv 0,3\Mod 4,  \\
         \sqrt2\, \tilde{S}(m,n,c), &\ m\equiv 1,2\Mod 4. 
    \end{array}\right.
\end{equation}
We will use the latter formula to obtain estimates on sums of Kloosterman sums $\tilde S$ in \S~\ref{sec:rademacher}.

\section{Rademacher-type formulas for the coefficients of \texorpdfstring{$g_m(\tau)$}{g\_m(tau)}}
\label{sec:rademacher}
Let $a_{m,n}$ and $\tilde{a}_{m,n}$ be the coefficients of the Fourier expansions of $g_m$ at $\infty$ and at $1$ respectively:
	\begin{align}
	g_m(\tau)-e^{-m\pi i \tau} &= \sum_{n\ge 1} a_{m,n}e^{n\pi i \tau}\,,\\
	(\tau/i)^{-3/2}g_m\Big(1-\frac1{\tau}\Big) &= \sum_{n\ge 0} \tilde{a}_{m,n}e^{2(n+3/8)\pi i \tau}\,. \label{eq:gm1expansion}
	\end{align}
The next result gives Rademacher-type formulas for these coefficients in terms of the Kloosterman sums from Section~\ref{sec:definitiions}. 

\begin{lemma} \label{prop:rademacher}
	Let $m$ be a positive integer. We have
	\begin{align}
	\label{eq:rademacher8}
    a_{m,n}&=\pi \ee(-\tfrac 38)\left(\frac nm\right)^{\frac 14}\sum_{c=1}^\infty \frac{S(-m,n,c)}{c}I_{1/2}\Big(\frac{2\pi\sqrt{m n}}{c}\Big),\qquad &&n\geq 1\,,\\
	%a_{m,n} &= m^{-1/2}\sum_{c>0} \frac{S(-m,n,c)}{c^{1/2}} \sinh\Big(\frac{2\pi\sqrt{mn}}{c}\Big) \,,\\
	\label{eq:rademacher1}
    \tilde{a}_{m,n}&=2\pi \left(\frac {2(n+3/8)}m\right)^{\frac 14}\sum_{\mathrm{odd}\;c>0} \frac{\tilde{S}(-m,n,c)}{c}I_{1/2}\Big(\frac{2\pi\sqrt{2m (n+3/8)}}{c}\Big),\quad &&n\geq 0. 
	%\tilde{a}_{m,n} &= (-1)^{m-1}2m^{-1/2}\sum_{\substack{c>0\\ c\in 1+2\Z}}\frac{\tilde{S}(-m/2,n+3/8,c)}{c^{1/2}} \sinh\Big(\frac{2\pi\sqrt{2m(n+3/8)}}{c}\Big)\,.
	\end{align}
\end{lemma}
\begin{proof}
    The technique is standard and essentially repeats the proof of Thm~3.2 and Prop.~5.1 in~\cite{JKK} (see also~\cite{DIT,BJO,Koh}), so we will only give a brief outline. Let the weight be $k=\frac 32$ and 
    \[\varphi_{m}(\tau,s)\coloneqq |4\pi m y|^{-k/2}M_{\mathrm{sgn}(m)k/2,\,s-1/2}(4\pi |m|y)\ee(mx)\]
    as in \cite[(2.15)]{JKK} ($M$ is the Whittaker $M$ function~\cite[Eq.~(3.14.2)]{DLMF}), which is an eigenfunction for the weight $3/2$ hyperbolic Laplace operator $\Delta_{3/2}$ with eigenvalue $(s-3/4)(1/4-s)$ and $\phi_{-m/2}(\tau,3/4)=e^{-m\pi i \tau}$. Consider the weak Maass--Poincar\'e series
	\begin{equation}
		P_{m}(\tau,s) \coloneqq  \sum_{\gamma\in\Gamma_{\infty}\backslash\Gamma_{\theta}}\nu_\theta(\gamma)^{-3}j(\gamma,\tau)^{-3/2}\varphi_{-m/2}(\gamma \tau,s), 
	\end{equation}
    which converges absolutely and uniformly on compact subsets of $\re s>1$. We will calculate its analytic continuation to $s= k/2= 3/4$ via its Fourier expansion. We have
    \begin{equation}\label{eq:PmTauS_FourierExpansion}
        P_m(\tau,s)=|2\pi m y|^{-3/4}M_{-3/4,\,s-1/2}(2\pi my)\ee(-mx/2)+\sum_{n\in \Z} a_m(n,y,s)\ee(nx/2), 
    \end{equation}
    where $a_m(n,y,s)$ is given by 
    %$\Gamma(2s)(2\pi m y)^{-\frac 34}\ee(-\tfrac 38)$ times 
    \begin{equation}\label{eq:FourierExpansionInfty_s}
        \frac{\Gamma(2s)\ee(-\tfrac 38)}{(2\pi m y)^{\frac 34}}\sum_{c=1}^\infty \frac{S(-m,n,c)}{c}\left\{\begin{array}{ll}
					\frac{\pi\sqrt{|m/n|}}{\Gamma(s+\frac 34)}W_{\frac 34,s-\frac 12}(2\pi ny) I_{2s-1}\big(\frac{2\pi\sqrt{mn}}c\big), & n>0,\\
					\frac{2\pi^{1+s}}{(2s-1)\Gamma(s-\frac 34)\Gamma(s+\frac  34)}\frac{m^s}{2^sc^{2s}y^{s-1}}, & n=0,\\
					\frac{\pi\sqrt{|m/n|}}{\Gamma(s-\frac 34)}W_{-\frac 34,s-\frac 12}(2\pi |n|y) J_{2s-1}\big(\frac{2\pi\sqrt{|mn|}}c\big),& n<0.
				\end{array}\right.
    \end{equation}
    Here $W$ is the Whittaker $W$ function~\cite[Eq.~(3.14.3)]{DLMF}.

    Using \cite[Prop.~7.1, Prop.~7.3, \& Prop.~7.8]{RadchenkoSun}, and noting that the above sum over $c$ in $\Z_{>0}$ decomposes into the sum of the two cases $2|c$ and $2\nmid \widetilde{c}$ from \cite[\S7]{RadchenkoSun}, we find that
    \begin{equation} \label{eq:kloostermanpartialsums}
        \sum_{1\leq c\leq x}\frac{S(-m,n,c)}c \ll_\eps  \left\{\begin{array}{ll}
            (|mn|^{1/4}+x^{1/6})|mnx|^\eps, &\ n\neq 0, \\
            m^{1/4+\eps}x^{1/6+\eps}, & \ n=0.
        \end{array}\right.
    \end{equation}
    By simple estimates for Bessel functions, we find that the series in  \eqref{eq:FourierExpansionInfty_s} is conditionally convergent for $\re s\ge 3/4$ (for details, see the argument in \cite[\S9.2]{SunUniform1}).

    The conditional convergence and estimates on $a_m(n,y,s)$ ensure that \eqref{eq:PmTauS_FourierExpansion} is absolutely convergent for $\re s\ge 3/4$. By analytic continuation, we can take $s=3/4$ in \eqref{eq:PmTauS_FourierExpansion}. Since $a_m(n,y,3/4)=0$ for $n\leq 0$, $P_m(\tau,3/4)$ is weakly holomorphic and equals $g_m(\tau)$ by uniqueness. This proves \eqref{eq:rademacher8}.
    
    The proof of the other formula follows similarly. For $\re s >1$, we have
    \begin{equation}
			(P_m(\cdot,s)|_{\frac 32}\sigma_1)(\tau)=\sum_{\gamma\in (\Gamma_\theta)_\infty\setminus \Gamma_\theta\sigma_1}(\nu_{\theta})_{\infty 1}^{-3}(\gamma) j(\gamma,\tau)^{-\frac 32} \phi_{- m/2}(\gamma \tau,s),
		\end{equation}
     which has Fourier expansion
    \begin{equation}\label{eq:PmTauS_FourierExpansion_cusp1_s}
		\Big(\frac{\tau}{\sqrt 2}\Big)^{-\frac 32}P_{m}\Big(1-\frac 2\tau ,s\Big)=\sum_{n\in \Z} \tilde{a}_m(n,y,s)e^{\pi i n_+ x},
	\end{equation}
    where $\tilde{a}_m(n,y,s)$ is given by $\Gamma(2s)(2\pi m y)^{-\frac 34}\ee(-\tfrac 38)$ times
    \begin{equation}\label{eq:FourierExpansion1_s}
        \sum_{\text{odd }c>0} \frac{\tilde{S}(-m,n,c)}{c/\sqrt 2}\left\{\begin{array}{ll}
				\frac{\pi|m/n_+|^{\frac 12}}{\Gamma(s+\frac 34)}W_{\frac 34,s-\frac 12}(2\pi n_+ y) I_{2s-1}\big(\frac{2\pi\sqrt{m n_+}}{c/\sqrt 2}\big), & n\geq 0,\\
				\frac{\pi|m/n_+|^{\frac 12}}{\Gamma(s-\frac 34)}W_{-\frac 34,s-\frac 12}(2\pi |n_+|y) J_{2s-1}\big(\frac{2\pi\sqrt{m |n_+|}}{c/\sqrt 2}\big),& n<0.\\
			\end{array}\right.
    \end{equation}
    (Recall that $n_+=n+3/8$ and $\sigma_1=(\smat{1/\sqrt 2}{\ -\sqrt 2}{1/\sqrt 2}0)$.) Since $8n+3$ cannot be $0$ or a negative square for $n$ in $\Z$, by combining \eqref{eq:RelationKLsumsTilde} with an argument similar to that in \cite[Prop.~7.3]{RadchenkoSun}, we find that
    \begin{equation} \label{eq:kloostermanpartialsumstilde}
        \sum_{1\leq c\leq x}\frac{\tilde{S}(-m,n,c)}c \ll_\eps (|mn_+|^{1/4}+x^{1/6})|mn_+x|^\eps \quad \text{for }n\in \Z.  
    \end{equation}
    We conclude the conditional convergence of \eqref{eq:FourierExpansion1_s} for $\tilde{a}_m(n,y,s)$ for $\re s \ge 3/4$, the fact that $\tilde{a}_m(n,y,3/4)=0$ for $n<0$, the absolute convergence of \eqref{eq:PmTauS_FourierExpansion_cusp1_s} for $\re s \ge 3/4$, and \eqref{eq:rademacher1} in the same way as in the proof of~\eqref{eq:rademacher8} above. 
\end{proof}

Note that~\eqref{eq:rademacher1} and~\eqref{eq:rademacher8} can be rewritten as
    \[\tilde{a}_{m,n} = \frac{2}{\sqrt{m}}\sum_{\mathrm{odd}\ c>0}\frac{\tilde{S}(-m,n,c)}{\sqrt{c}}\sinh\Big(\frac{2\pi\sqrt{2m(n+3/8)}}{c}\Big)\]
and
    \[a_{m,n} = \frac{\ee(-\tfrac{3}{8})}{\sqrt{m}}\sum_{c=1}^{\infty}
    \frac{S(-m,n,c)}{\sqrt{c}}\sinh\Big(\frac{2\pi\sqrt{mn}}{c}\Big).\]
From these identities, together with~\eqref{eq:kloostermanpartialsums} and~\eqref{eq:kloostermanpartialsumstilde}, we get the following corollary (the big $O$ term is not optimal, but will suffice for our purposes).
\begin{lemma}
For any $k,m\ge1$, we have
    \begin{align*}
    \tilde a_{m,n} = \frac{2}{\sqrt{m}}
    \sum_{\substack{1\le c<2k+1\\ c\ \mathrm{ odd}}}
    \frac{\tilde{S}(-m,n,c)}{\sqrt{c}}\sinh\Big(\frac{2\pi\sqrt{2mn_+}}{c}\Big)
     + O\bigg(m^{3/2}n_+^2\exp\Big(\frac{2\pi\sqrt{2mn_+}}{2k+1}\Big)\bigg)
    \end{align*}
and if also $n\ge1$, then
    \begin{align*}
    a_{m,n} = \frac{\ee(-\tfrac{3}{8})}{\sqrt{m}}
    \sum_{1 \le c < k+1}
    \frac{S(m,n,c)}{\sqrt{c}}\sinh\Big(\frac{2\pi\sqrt{mn}}{c}\Big)
     + O\bigg(m^{3/2}n^2\exp\Big(\frac{2\pi\sqrt{mn}}{k+1}\Big)\bigg).
    \end{align*}
\end{lemma}
We will only use the case $k=1$, from which we conclude that for $m\ge1$
    \begin{equation} \label{eq:atildeapprox}
    \tilde a_{m,n} = \frac{(-1)^n}{\sqrt{m}}e^{2\pi\sqrt{2m(n+3/8)}}
     + O\bigg(m^{3/2}(n+3/8)^2\exp\Big(\frac{2\pi\sqrt{2m(n+3/8)}}{3}\Big)\bigg)
    \end{equation}
and for $m,n\ge1$
    \begin{equation} \label{eq:aapprox}
    a_{m,n} = \frac{e^{2\pi\sqrt{mn}}}{2\sqrt{m}} 
            + O\bigg(m^{3/2}n^2e^{\pi\sqrt{mn}}\bigg).
    \end{equation}

\section{Approximate formula for \texorpdfstring{$f_n$}{f\_n} and the special function \texorpdfstring{$\Phi(z)$}{Phi(z)}}
\label{sec:Phi}

In this section, we will use~\eqref{eq:atildeapprox} to obtain a simple and precise approximation of $f_n(x)$ for $x>n/9$. This essentially reduces the study of $f_n$ in this range to the study of a new special function $\Phi(z)$, defined in~\eqref{eq:phidef} below. We will show that $\Phi(z)$ extends to a meromorphic function in $\C$ with a single simple pole at $z=0$, and then we will give estimates for the values $\Phi(x)$ for negative real $x$ and show that $\Phi$ has infinitely many negative real zeros.

\subsection{Approximate formula for \texorpdfstring{$f_n(x)$}{f\_n(x)} for \texorpdfstring{$x > n/9$}{x>n/9}}
\label{sec:approx-cn}

We start by showing how~\eqref{eq:atildeapprox} leads to an approximate formula for $f_n(x)$ valid for $x>n/9$. The idea is to start with~\eqref{eq:fndef2}, which we rewrite in the form
    \[\frac{f_n(x)}{\sin(\pi x)} = 
    \int_{0}^{\infty}g_n(1+it)e^{-\pi x t} dt\,,\qquad x>n\,,\]
and then use the Fourier expansion of $(\tau/i)^{-3/2}g_{n}(1-1/\tau)$. From
	\[\int_{0}^{\infty}e^{-2\pi \alpha t^{-1}-\pi x t}t^{-3/2}dt
	= \frac{e^{-2\pi\sqrt{2\alpha x}}}{\sqrt{2\alpha}}\,,\qquad \re\alpha,\, \re x>0,\]
together with~\eqref{eq:gm1expansion}, we obtain the identity
	\[\frac{f_n(x)}{\sin(\pi x)} =
	\sum_{\nu\ge0}\tilde a_{n,\nu}
	\frac{e^{-2\pi \sqrt{2x(\nu+3/8)}}}{\sqrt{2(\nu+3/8)}}\,,\qquad x>n.\]
By~\eqref{eq:atildeapprox},
    \begin{equation} \label{eq:fnapprox}
    \frac{f_n(x)}{\sin \pi(x-n)} = \sum_{\nu\ge0}\frac{e^{2\pi\sqrt{2(\nu+3/8)}(\sqrt{n}-\sqrt{x})}}{\sqrt{2n(\nu+3/8)}}
	+E_n(x)\,,\qquad x > n\,,
    \end{equation}
where
	\[E_n(x):=\sum_{\nu \ge0 }\Big(\tilde a_{n,\nu} - \frac{(-1)^{n}}{\sqrt{n}}e^{2\pi\sqrt{2n(\nu+3/8)}}\Big)\frac{e^{-2\pi \sqrt{2x(\nu+3/8)}}}{\sqrt{2(\nu+3/8)}}\,.\]
Note that from the error term in~\eqref{eq:atildeapprox} it follows that $E_n(x)$ is well-defined and analytic for $\re \sqrt{x} >\sqrt{n}/3$, and in this region it satisfies
    \[E_n(x)=O\Big(\frac{e^{-\pi\sqrt{3}\re(\sqrt{x}-\sqrt{n}/3)}}{|\sqrt{x}-\sqrt{n}/3|^5}\Big)\,.\]
We define
    \begin{equation}\label{eq:phidef}
    \Phi(z) := \sum_{\nu=0}^{\infty}
    \frac{e^{-2\pi z\sqrt{2(\nu+3/8)}}}{\sqrt{2(\nu+3/8)}}\,,\qquad \re z>0.
    \end{equation}
In Section~\ref{sec:phiac} below, we will show that $\Phi(z)$ continues analytically to a meromorphic function in $\C$ with a simple pole at $z=0$ and no other singularities. Then~\eqref{eq:fnapprox} together with uniqueness for analytic functions gives the following result.
\begin{lemma} \label{prop:fnapprox}
For $z \ne n$ with $\re z>0$ and $\re\sqrt{z}>(1/3+\eps)\sqrt{n}$, we have
	\begin{equation} \label{eq:fn9}
		\frac{f_n(z)}{\sin\pi(z-n)} = \frac{\Phi(\sqrt{z}-\sqrt{n})}{\sqrt{n}} + E_n(z)\,,
	\end{equation}
    where $E_n(z)$ is analytic in the above range and satisfies
    \[E_n(z) = O_{\eps}\big(e^{\pi\sqrt{3}(\sqrt{n}/3-\re\sqrt{z})}\big)\,.\]
\end{lemma}
To extract more information about $f_n$ from~\eqref{eq:fn9}, we turn to a more detailed study of the special function~$\Phi$. 

\subsection{Analytic continuation of \texorpdfstring{$\Phi$}{Phi}}
\label{sec:phiac}
From~\eqref{eq:phidef}, we compute the Mellin transform of $\Phi$ as
    \[\int_{0}^{\infty}\Phi(t)t^{s-1}dt 
      = (2\pi)^{-s}2^{-(s+1)/2}\Gamma(s)\zeta\Big(\frac{s+1}{2},\frac{3}{8}\Big),\]
where $\zeta(s,x)=\sum_{n\ge1}(n+x)^{-s}$ is the Hurwitz zeta function. A standard argument starting with the inverse Mellin transform followed by pushing the abscissa of the line of integration to $-\infty$ and collecting the residues, then gives the following alternative expression for $\Phi$:
    \begin{equation} \label{eq:phitaylor}
    \Phi(z) = \frac{1}{2\pi z}+\sum_{k=0}^{\infty}\frac{2^{(k-1)/2}\zeta(\frac{1-k}{2},\frac38)(-2\pi z)^k}{k!}\,.
    \end{equation}
This readily shows that $\Phi$ is meromorphic in $\C$ with a simple pole at $0$ and no other singularities. Our finer analysis of the properties of $\Phi$ is based on the following functional equation.
    \begin{lemma}\label{prop:Psi}
    For $z$ in $\C^{\times}$, we have
    \begin{equation} \label{eq:phifeq}
    \Phi(z)+\Phi(-z) = \sum_{n\ge1}\frac{2\cos(\pi(\frac{3n-1}{4}-\frac{z^2}{n}))}{\sqrt{n}}\,.\end{equation}
    \end{lemma}
    \begin{proof}
    The left-hand side of~\eqref{eq:phifeq} is clearly an even entire function by~\eqref{eq:phitaylor}. Rewriting the right-hand side as
    \[2\sum_{n\ge1}\frac{\cos(\pi\frac{3n-1}{4})}{\sqrt{n}}+
      2\sum_{n\ge1}\frac{\cos(\pi(\frac{3n-1}{4}-\frac{z^2}{n}))-\cos(\pi\frac{3n-1}{4})}{\sqrt{n}}\]
    and using absolute uniform convergence on compacts for the second series, we get that the right-hand side of~\eqref{eq:phifeq} is also even and entire. Therefore, it suffices to check that both sides of~\eqref{eq:phifeq} have the same Taylor coefficients. By~\eqref{eq:phitaylor}, the coefficient of $z^{2m}$ of the left-hand side of~\eqref{eq:phifeq} is
    \[A_m=\frac{(-2\pi)^{2m}2^{m+1/2}\zeta(1/2-m,3/8)}{(2m)!}.\]
    For the right-hand side of~\eqref{eq:phifeq}, the $2m$-th Taylor coefficient is
    \[B_m=\frac{\pi^m}{m!}\sum_{n\ge1}\Big(\ee(-\tfrac{m}{4})
        \frac{\ee(\frac{3n-1}{8})}{n^{m+1/2}}
       +\ee(\tfrac{m}{4})\frac{\ee(-\frac{3n-1}{8})}{n^{m+1/2}}\Big).\]
    The equality $A_m=B_m$ can then be easily derived from Hurwitz's formula
	\[\zeta(1-s,a)=\frac{\Gamma(s)}{(2\pi)^s}\Big(\ee(-\tfrac{s}{4})\sum_{n\ge1}\tfrac{\ee(na)}{n^s}+\ee(\tfrac{s}{4})\sum_{n\ge1}\tfrac{\ee(-na)}{n^s}\Big),\]
    that holds for $0<a<1$, $\re s>0$, together with the duplication formula for the gamma function.
    \end{proof}
Since $\Phi(z)$ decays exponentially as $z\to+\infty$,~\eqref{eq:phifeq} shows that to understand the behavior of $\Phi(z)$ for large negative $z$, it suffices to understand the behavior of the simpler function
    \begin{equation} \label{eq:psidef}
    \Psi(x):=\sum_{n\ge1}\frac{2\cos(\pi(\frac{3n-1}{4}-\frac xn))}{\sqrt{n}}\,.
    \end{equation}

\subsection{Estimates of \texorpdfstring{$\Phi(x)$}{Phi(x)} for \texorpdfstring{$x<0$}{x<0}}
    If we had square root cancellation in~\eqref{eq:psidef}, $\Psi(x)$  would grow more slowly than any positive power of $x$. Falling short of proving that, we have the following weaker result.
    \begin{theorem}\label{thm:lindelof}
    We have 
    \[ \Psi(x)\ll  |x|^{0.108} \]
    when $x\to +\infty$.
    \end{theorem}
    
    We require the following classical result which follows from \cite[Lem. 4.2, Lem. 4.8]{T}.
    \begin{lemma}\label{lem:titsch}
    Suppose that  $f$ is a real valued function in $C^1([a,b])$ and that $f'$ is monotonic. If $0<m\le |f'(t)|\le \frac23$ for all $t$ in $[a,b]$, then
    \[ \sum_{a<n\le b} \ee(f(n)) \ll \frac{1}{m}. \]  
    \end{lemma}
    \begin{proof}[Proof of Theorem~\ref{thm:lindelof}] We note that $\Psi(x)/2$ is the real part of
   \begin{equation} \label{eq:sum}  \Theta(x)\coloneqq \sum_{n\ge 1}\frac{\exp(i \pi (\tfrac{3n-1}{4}-\tfrac{x}{n}))}{\sqrt{n}},\end{equation}
and our goal will therefore be to estimate the size of $\Theta(x)$ for large $x$. 

We begin with the estimation of the contribution to the sum from large $n$. Setting
    \[  \Theta_y(x)\coloneqq \sum_{n> y}\frac{\exp(i \pi (\tfrac{3n-1}{4}-\tfrac{x}{n}))}{\sqrt{n}}\]
and
    \[ F_y(t)\coloneqq \sum_{y<n\le t} \frac{\exp(\tfrac{i\pi(3n-1)}{4})}{\sqrt{n}},\]
    we find by partial summation that
    \begin{equation} \label{eq:largen} \Theta_y(x)=i\pi x \int_y^\infty \frac{F_y(\infty)- F_y(t)}{t^2} e^{-i\pi \frac{x}{t}} dt . \end{equation}
    Hence, since $F_y(\infty)- F_y(t)\ll 1/\sqrt{t}$, $\Theta_y(x)$ is trivially bounded when $y= x^{\frac23}$. The remaining problem is therefore to estimate
    \[  R(x)\coloneqq \sum_{1\le n \le x^{\frac23} }\frac{\exp(i \pi (\tfrac{3n-1}{4}-\tfrac{x}{n}))}{\sqrt{n}} . \]
 
 By periodicity of the exponential $\exp(i \pi \tfrac{(3n-1)}{4})$, it suffices to estimate sums of the form
 \[ \sum_{ k \le x^{\frac23}} \frac{\exp(i\pi \frac{x}{8k+q})}{\sqrt{8k+q}}, \]   
 where $0\le q<8$. We now show that
 \begin{equation} \label{eq:large} \sum_{x^{\frac12} < k \le x^{\frac23}} \frac{\exp(i\pi \frac{x}{8k+q})}{\sqrt{8k+q}} \ll 1.\end{equation}
 Indeed, setting
 \[ G_x(t)\coloneqq \sum_{x^{\frac12}< k \le t} \exp(i\pi \frac{x}{8k+q}), \]
 we get by partial summation that
 \[ \sum_{x^{\frac12} < k \le x^{\frac23}} \frac{\exp(i\pi \frac{x}{8k+q})}{\sqrt{8k+q}} \ll \frac{G_x(x^{\frac23})}{x^{\frac13}}+\int_{x^{\frac12}}^{x^{\frac23}} \frac{G_x(t)}{t^{\frac32}} dt.\] 
 By Lemma~\ref{lem:titsch}, we have $|G_x(t)|\ll t^2/x$ which yields the desired bound \eqref{eq:large}.
  
Our remaining task is to find a bound for 
   \[ \sum_{0 \le k \le x^{\frac12}} \frac{\exp(i\pi \frac{x}{8k+q})}{\sqrt{8k+q}}. \]   
We split the sum dyadically and find that
  \[ \sum_{y < k \le 2y} \frac{\exp(i\pi \frac{x}{8k+q})}{\sqrt{8k+q}} \ll  x^{\kappa} y^{\lambda-\frac12-2\kappa},\]
where $(\kappa,\lambda)$ is an exponent pair. Using the pair $(\frac{4742}{38463}+\varepsilon,\frac{35371}{51284}+\varepsilon)$ in the range $x^{0.27}<y \le x^{\frac12}$ and the pair $(\frac{18}{199}+\varepsilon,\frac{593}{796}+\varepsilon)$ when $1\le y \le x^{0.27}$ (see \cite[Lemma 1.1]{TY}), we obtain the desired result.
\end{proof}

The following result yields the lower bound in \eqref{eq:sqrsum}. 
\begin{theorem} \label{thm:moment}
  We have
  \[ \int_{T}^{2T} |\Psi(x)|^2dx = T \log T + O(T)\]
when $T\to \infty$.
\end{theorem}
\begin{proof}
The proof of Theorem~\ref{thm:lindelof} shows that $\Theta_{\sqrt{T}}(x)$ is a bounded function on $[-2T,2T]$. Setting
\[ E_{\sqrt{T}}(x)\coloneqq \sum_{n\le \sqrt{T}} \frac{\exp(i \pi (\tfrac{3n-1}{4}-\tfrac{x}{n}))}{\sqrt{n}}, \] we therefore have
 \begin{equation} \label{eq:splitint} \int_{T}^{2T} |\Psi(x)|^2dx = \int_{T}^{2T}|E_{\sqrt{T}}(x)|^2 dx - 2\operatorname{Re}\int_T^{2T} E_{\sqrt{T}}(x)\overline{\Theta_{\sqrt{T}}(x)}dx+O(T).  \end{equation}
 Opening the square and integrating termwise, we see by invoking Hilbert's inequality that the first term on the right-hand side of \eqref{eq:splitint} is $T\log T+ O(T) $. The middle term is $O(T)$. Indeed, since both $E_{\sqrt{T/2}}-E_{\sqrt{T}}$ and $\Theta_{\sqrt{T}}-\Theta_{\sqrt{2T}}$ are $\ll 1$, it suffices to observe that
 \[ \int_T^{2T} E_{\sqrt{T/2}}(x)\overline{\Theta_{\sqrt{2T}}(x)}dx
 \ll  \sum_{n\le \sqrt{T}} \sqrt{n} \ll T^{\frac34}, \] 
which follows by termwise integration and a variation of the first part of the proof of Theorem~\ref{thm:lindelof}.  \end{proof}

\subsection{Extraneous zeros in the right half-plane} 
We will now prove items (i) and (iii) of Theorem~\ref{thm:zeros}. We begin
by establishing a lemma which immediately yields item (i) and which is also the key ingredient in the proof of (iii). We set 
 \[  \Delta(t_1,t_2,n)\coloneqq \{z: \  0<\operatorname{Re}z\le n, \  t_1 < \operatorname{Re} (\sqrt{z}-\sqrt{n})^2 < t_2, |\operatorname{Im} (\sqrt{z}-\sqrt{n})^2|< \log n \} \]
and let $N(f;t_1,t_2,n)$ be the number of zeros of $f$ in $\Delta(t_1,t_2,n)$, with multiplicities accounted for  in the usual way.
\begin{lemma}\label{lem:regularzeros}
When $n$ is large enough and $3\le t_1<t_2 \le n/2$, we have
\begin{equation} \label{eq:Nzeros}
     |N(f_n;t_1,t_2,n)-(t_2-t_1)|\le 4 .
     \end{equation}
\end{lemma}

\begin{proof}
We set
\[ \varphi(z)\coloneqq (\sqrt{z}-\sqrt{n})^2 \]
and observe that $\varphi$ maps $ \Delta(t_1 ,t_2,n)$ to the rectangle 
\[ R(t_1,t_2,n)\coloneqq \{ w: \ t_1< \operatorname{Re} w < t_2, |\operatorname{Re} w|< \log n\}. \]
The function
\[ F(w)\coloneqq f_n(\varphi^{-1}(w))/\sin (\pi \varphi^{-1}(w)) \]
has then $N(f_n;t_1,t_2,n)$ zeros in $R(t_1,t_2,n)$. 

    Let $m_1$ and $m_2$ be integers such that 
    \[ m_1-\tfrac12\le t_1 < m_1+\tfrac12 \quad \text{and} \quad  m_2-\tfrac12 < t_2 \le m_2+\tfrac12.\]
    We write $w=u+iv$ and $F=U+iV$ and let $k$ be an integer. Then since
    \[ \cos\pi(k+iv)=\cos \pi k \cosh \pi v  ,\]
    we find from Lemma~\ref{prop:fnapprox} and Lemma~\ref{prop:Psi} along with the series representation \eqref{eq:psidef}  that $\frac{d^2}{dy^2}U(k+iy)$ has constant sign. This means that $U$ itself can change sign at most 2 times along any segment of the vertical line $y\mapsto k+iy$. We now apply the argument principle in the two rectangles $R_n(m_1-\frac12,m_2+\frac12)$ and $R_n(m_1+\frac12,m_2-\frac12)$. Repeating the arguments that led to Theorem~\ref{thm:lindelof},
    we find that
    \[ F(u\pm i \log n)=2\cos\pi\big(\tfrac12-i\log n -u \big)(1+o(1)),\]
    which means that the total change of argument of $F$ along the horizontal parts of the boundaries of these rectangles is respectively
    \[ 2\pi(m_2-m_1+1)+o(1) \quad \text{and} \quad 2\pi(m_2-m_1-1)+o(1). \] 
    If $F$ has no zeros on the vertical parts of the boundary of the rectangles, %$R(m_1-\frac12,m_2+\frac12,n)$ or $R(m_1+\frac12,m_2-\frac12,n)$
    then in either case we can have at most an additional change of the argument of $4\pi$. By making, if necessary, a slight shift of the vertical sides of any of the two rectangles, we may in fact assume that $F$ has no zeros on these vertical line segments, and so we conclude that the total change of argument lies between
    \[ 2\pi(m_2-m_1-3) \quad \text{and} \quad 2\pi(m_2-m_1+3). \]
    By the argument principle and these observations, we arrive at \eqref{eq:Nzeros}.
\end{proof}

\begin{proof}[Proof of item (iii) of Theorem~\ref{thm:zeros}]
We retain the notation from the preceding proof. We fix a large $M$ and set 
\[ E(M)\coloneqq \{ 3\le u \le x: |F(u)|\ge M \}.\]
By Theorem~\ref{thm:moment} and Theorem~\ref{thm:lindelof}, we have then
\[ x\log x \ll_M x + |E(M)| x^{0.216} \]
which implies that
\[ |E(M)|\gg x^{0.784} \]
if $T$ is large enough. The set
\[ \widetilde{E}(M)\coloneqq \bigcup_{u\in E(M)} (u-3,u+3) \]
is a union of open intervals $I$, each of length $\ge 6$, and $|\Psi(x)|\ge M/2$ for $x$ in $\widetilde{E}(M)$ if $M$ is large enough. There will be $\ge |I|-4$ nonreal zeros in
$\{z=x+iy: x\in I, |y|\le \log n\}$ and so the number of nonreal zeros of $F$ in $R(3,x,n) $ is $\gg |E(M)| \gg x^{0.784}$.
\end{proof}

\subsection{Positive extraneous zeros}
In this section, we will prove part (ii) of Theorem~\ref{thm:zeros}. We will do this by showing that the size of the integrals of $\Psi$ and $|\Psi|$ differ over suitably long intervals. This will in turn, by Lemma~\ref{prop:fnapprox} and Lemma~\ref{prop:Psi}, imply a discrepancy between the corresponding integrals of $f_n$ and $|f_n|$ and hence the existence of at least one real zero. 

We begin with the following lemma.
\begin{lemma}\label{lem:integral} We have  
	\begin{equation}\label{eq:Psishort} \int_{T}^{T+V} \Psi(x) dx \ll T^{\frac12} \log T \end{equation}
	when $\sqrt{T}\log^2 T\le V \le T$ and $T$ is large enough.
	\end{lemma}
	\begin{proof}
    It will be convenient to prove instead that
    \[ \int_{T}^{T+V} \Theta(x) dx \ll T^{\frac12} \log T\]
    which clearly implies \eqref{eq:Psishort} since $\Psi/2$ is the real part of $\Theta$. 
	We begin by noting that in view of \eqref{eq:largen}, $\Theta_y(x)\ll \frac{x}{y^{\frac32}}$. This entails that
	\[ \int_T^{T+V} \Theta_T(x) dx \ll V T^{-\frac12}. \]
	We split the remaining sum $\Theta(x)-\Theta_T(x)$ from \eqref{eq:sum} dyadically and consider
	\begin{equation} \label{eq:dyadint}  S_y(x)\coloneqq \sum_{y<n\le 2y}\frac{\exp(i \frac{\pi}{4} (\tfrac{3n-1}{4}-\tfrac{x}{n}))}{\sqrt{n}}. \end{equation}
	Integrating term by term, we find that
	\begin{equation} \label{eq:term} \Big|\int_{T}^{T+V} S_y(x) dx\Big| \le \frac{2}{\pi} \max_{T\le x \le T+V} \Big|\sum_{y<n\le 2y}\sqrt{n} \exp(i \pi (\tfrac{3n-1}{4}-\tfrac{x}{n}))\Big|. \end{equation}
	Using again Lemma~\ref{lem:titsch}, we find that
	\begin{equation} \label{eq:b2}  \int_{T}^{T+V} S_y(x) dx \ll \frac{y^{\frac52}}{T} \end{equation}
	when $y\ge \sqrt{T}$. However, setting
	\[ G_1(t)\coloneqq \sum_{y<n\le t} \exp(i \pi (\tfrac{3n-1}{4})), \]
	we get by applying partial summation to the right-hand side of \eqref{eq:term},
	\[ \int_{T}^{T+V} S_y(x) dx \ll x \int_{y}^{2y} G_1(t) \frac{e^{-i\pi \frac{x}{t}}}{t^{\frac32}} dt+O(\sqrt{y}). \]
	We observe that $G_1$ is an $8$-periodic function. Now defining inductively 
	\[ G_k(t)\coloneqq \int_{y}^t G_{k-1}(u) du - \int_{y}^{y+8} G_{k-1}(u) du ,\]
	we get by repeated integration by parts,
	\begin{equation} \label{eq:b1}   \int_{T}^{T+V} S_y(x) dx \ll x^k \int_{y}^{2y} G_k(t) \frac{e^{-i\pi \frac{x}{t}}}{t^{2k-\frac12}} dt+O(\sqrt{y})\ll \frac{T^k}{y^{2k-\frac32}}+ O(\sqrt{y}). \end{equation}
The upper bounds in \eqref{eq:b2} and \eqref{eq:b1} coincide when
	\[ y=T^{\frac{k+1}{2k+1}}, \] and we therefore get
	\[  \int_{T}^{T+V} (\Theta_{\sqrt{T}}(x)-\Theta_T(x))dx \ll T^{\frac12}\]
	if we choose $k=2$. (We could get the exponent $\frac14+\varepsilon$ by choosing $k$ large.)
	
	Finally, for $y\le \sqrt{T}$, picking a suitable exponent pair $(\kappa,\lambda)$, we see from \eqref{eq:dyadint} that
	\[  \int_{T}^{T+V} S_y(x) dx \ll T^{\kappa} y^{\lambda+\frac12-2\kappa}. \]
	Choosing $\kappa=\lambda=\frac12$, we therefore get
	\[ \int_{T}^{T+V} (\Theta(x)-\Theta_{\sqrt{T}}(x))dx \ll \sqrt{T} \log T. \qedhere \] 
	\end{proof}
	
	On the other hand, we also have the following.
	\begin{lemma} \label{lem:short}
	There exists a positive number $M$ such that 
	\[ \int_{T}^{T+M} |\Psi(x)| \, dx \ge 1 \]
	for all real numbers $T$.
	\end{lemma} 
\begin{proof}
Using the fact that
\[ \Psi''(x)=\sum_{n\ge1}\frac{2\pi^2 \cos(\pi (\tfrac{3n-1}{4}-\tfrac{x}{n}))}{n^{\frac52}}\]
and Rouch\'{e}'s theorem, we find that the zeros of $\Psi''$ are real and simple. Moreover, this zero set is uniformly discrete, and there is one zero in the interval $(k,k+1)$ for every integer $k$.
For two consecutive zeros $t_{k}$ and $t_{k+1}$, let $f$ be the polynomial of degree $4$ that vanishes along with its derivative at both $t_{k}$ and $t_{k+1}$ with $f\Psi''>0$  on 
$(t_k,t_{k+1})$. Then
\[ \int_{t_k}^{t_{k+1}} \Psi''(x)f(x) dx= \int_{t_k}^{t_{k+1}} \Psi(x)f''(x)dx \le \max_{x\in [t_k,t_{k+1}]} |f''(x)| \int_{t_k}^{t_{k+1}} |\Psi(x)|\, dx. \]
Since the sequence $t_k$ is uniformly discrete and $|\Psi''|$ is uniformly bounded below at the integers, the integral to the left is uniformly bounded below and the maximum to the right is uniformly bounded above. 
\end{proof}
The proof of item (ii) of Theorem~\ref{thm:zeros} is now immediate. Indeed, it suffices to divide the interval $[x/2,x]$ into $\gg \sqrt{x}/\log^2 x$ subintervals so that Lemma~\ref{lem:integral} applies to each of them. Appealing again to Lemma~\ref{prop:fnapprox} and Lemma~\ref{prop:Psi} along with Lemma~\ref{lem:short}, we see that each of the $\gg \sqrt{x}/\log^2 x$ subintervals must contain at least one zero of $f_n$.

\section{Further estimates for positive arguments of \texorpdfstring{$f_n$}{f\_n}} \label{sec:positivesmall}
This section gives the proof of Theorem~\ref{thm:positiverealsmall} and the upper bound in \eqref{eq:sqrsum} of Theorem~\ref{thm:sqrsum}. We will resort to alternate asymptotic expressions for $f_n$ beyond $\re \sqrt{z} > (1/3+\varepsilon) \sqrt{n} $, where quantitative information about $f_n$ is less accessible.

We will obtain the bound~\eqref{eq:outsidepw} for $f_n(x)$ via
	\[f_n(x)=e^{c \pi n} \int_{-1}^{1} F(t+ic,x)  e^{-i\pi n t} dt,\]
where $F(\tau,z)$ is the generating function for $f_n$ (see~\eqref{eq:gfseries}, \eqref{eq:fngf}). To prove Theorem~\ref{thm:sqrsum}, instead of using this fact directly, we will apply a duality argument along with a general method from~\cite{BRS} for estimating the Fourier coefficients of functions of the form
	\[ F_{\varphi}(\tau):= \frac{1}{2} \int_{-1}^1 K(\tau,w) \varphi(w) dw \]
when $\varphi$ is a function of moderate growth in $\HH$.

\subsection{Behavior of \texorpdfstring{$f_n(x)$}{f\_n(x)} when \texorpdfstring{$0\le x\le n/8$}{0<x<n/8}}\label{subsec:pos}

This section gives the proof of the pointwise bound \eqref{eq:outsidepw} in the range $0\le x \le n/8$. The proof technique relies on \cite[Sec. 6]{BRS}.

We begin by recalling the setup from \cite[Sec. 6]{BRS}. We let $\gamma_{\tau}$ be the (generically unique) element of $\Gamma_{\theta}$ that maps $\tau$ in $\HH$ to the
fundamental domain~$\bF$. We denote by $\bH(\tau)$ the imaginary part of $\gamma_{\tau}\tau$, i.e.,
    \[\bH(\tau)=\im\gamma_{\tau}\tau .\]
It is easy to see that the function $\bH\colon\HH\to\R$ is continuous and $\Gamma_{\theta}$-invariant. We define $\bN\colon\HH\to\Z_{\ge0}$ as one plus
the number of inversions~$S$ that appear in the canonical representation
of $\gamma_{\tau}$, i.e.,
    \[ \bN(\tau) = j+\eps_0+\eps_1  ,\qquad  \gamma_{\tau}=
    S^{\eps_0}T^{2m_1}ST^{2m_2}\dots ST^{2m_j}S^{\eps_1}.\]
In cases when $\gamma_{\tau}$ is not uniquely defined, i.e., when $\tau$ is in $\Gamma_{\theta}\partial\bF$, we let $\bN(\tau)$ be the larger of the two possible values.

We need three lemmas. The first is a slight variation of \cite[Prop. 6.1]{BRS}. Consider the domain~$\mathcal{D}$ illustrated in Figure~\ref{fig:crescent}.
Explicitly we set
    \[\mathcal{D}\coloneqq\{\tau\in \HH\colon |\re\tau|<1,\; \sqrt{3/4}<|\tau|<\sqrt{4/3},\; |\tau\pm 1/2|>1/2\};\]
the particular shape of~$\mathcal{D}$ is not important as long as it contains the geodesic from~$-1$ to~$1$ and lies in $\ol{\bF\cup S\bF}$. Also let $\ell$ be the contour illustrated in Figure~\ref{fig:jell}. The essential feature of $\ell$ is that it has a positive angle with the real line at $0$. By $J^{-1}(\ell)\subset\ol{\bF}$ we denote the preimage of the path $\ell$ under the conformal map $J|_{\bF}$ (see Figure~\ref{fig:jell}).

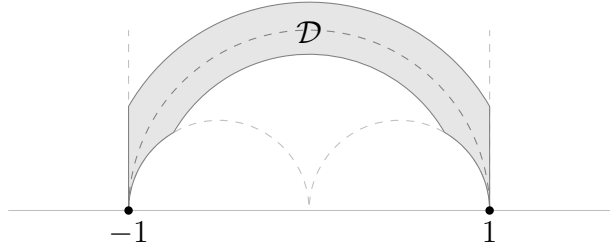
\begin{figure}[ht]
    \centering
    \begin{tikzpicture}
        \definecolor{cv0}{rgb}{0.95,0.95,0.95}
        \definecolor{cv1}{rgb}{0.90,0.90,0.90}
        \clip(-9,-0.5) rectangle (7,3);
        \begin{scope}[scale=0.8,xshift=-1cm]
            \draw[lightgray] (-5,0) -- (5,0);
            \draw[lightgray,dashed] (-3,0) arc  (180:0:1.5);
            \draw[lightgray,dashed]  (3,0) arc  (0:180:1.5);
            \draw[lightgray,dashed]  (-3,0) -- (-3,3);
            \draw[lightgray,dashed]  (3,0) -- (3,3);
            %\fill[color=cv0] (-3,0) arc (180:0:3) -- (3,6.5) -- (-3,6.5);
            \fill[color=cv1] (-3,0) -- (-3,1.7320) arc (150:30:3.4641) -- (3,0) arc (0:60:1.5) arc (30:150:2.5980) arc (120:180:1.5);
            \draw[gray] (-3,0) -- (-3,1.7320) arc (150:30:3.4641) -- (3,0) arc (0:60:1.5) arc (30:150:2.5980) arc (120:180:1.5);
            \draw[gray,dashed] (-3,0) arc  (180:0:3);
            %\draw[lightgray,dashed] (0,3)  --  (0,6.5);
            \draw (0,2.6) node[above]{$\mathcal{D}$};

            %\fill[black] (0,3) circle (0.07) node[below] {$i$};
            \fill[black] (-3,0) circle (0.07) node[below] {$-1$};
            \fill[black] (3,0) circle (0.07) node[below] {$1$};

            %\draw[color=blue, line width=0.3mm] (-3,0) to[out=90,in=180] (-1,2) to[out=10,in=225] (0,3) to[out=45,in=160] (1.6,3.4) to[out=340,in=90] (3,0);
            %\draw (0,3) node[above]{$J^{-1}(\ell)$};

            %\draw (-4,4.5) node[above]{$\HH$};
        \end{scope}
    \end{tikzpicture}
    \caption{The domain $\mathcal{D}$.}
    \label{fig:crescent}
\end{figure}

\begin{lemma}\label{lem:ker}
    There exists a constant $C$ depending only on $\ell$ and $q$ such that
    \[ \int_{J^{-1}(\ell)} |K(\tau, z)|^{q} |dz| \le C \max(1,(\im \tau)^{-2(q-1)}). \]
    when $\tau$ is in $\bF$ and $q\ge 1$.
\end{lemma}
\begin{proof}
    Without loss of generality we may assume that $\re\tau$ is in $(-1,0)$,
    $|\tau|>1$ and that $\im \tau \le 1/2$. To estimate the integral, we change the variable of integration to $w=J(z)$. We let $t(w)$ be the inverse function to $z\mapsto J(z)$. Using~\eqref{eq:theta4dz} to make the change of variables, we obtain
    \[ \int_{J^{-1}(\ell)} |K(\tau, z)|^q |dz| =
    \pi^{-1} |\theta(\tau)|^q
    \int_{\ell}\frac{|\theta(t(w))|^{3q-4} |w|^{q-1} |1-64/w|^{(q-1)/2}} {|J(\tau)-w|^q} |dw|.
    \]
    Since $J(\tau)$ is in the upper half-plane and $w$ is in the lower half-plane, and also since $\ell$ has a positive angle with the real line at $0$,
    we have $|J(\tau)-w|\gg \sqrt{|J(\tau)|^2+|w|^2}$, whence
    \[ \int_{J^{-1}(\ell)} |K(\tau, z)|^q |dz| \ll
    |\theta(\tau)|^q
    \int_{\ell}\frac{|\theta(t(w))|^{3q-4} |w|^{(q-1)/2}}{(|J(\tau)|^2+|w|^2)^{q/2}} |dw|.
    \]
    Now
    \begin{align*} \int_{\ell}\frac{|\theta(t(w))|^{3q-4} |w|^{(q-1)/2}} {(|J(\tau)|^2+|w|^2)^{q/2}} |dw| & \ll 1+ \int_0^{\frac{1}{2}} \frac{t^{7q/8-1} \left(\log \frac{1}{t}\right)^{(3q-4)/2}} {(|J(\tau)|^2+t^2)^{q/2}} dt \\
        & \ll |J(\tau)|^{-q/8} \left(\log \frac{1}{|J(\tau)|}\right)^{-q/2+2(q-1)}, \end{align*}
    from which the result follows. (Here we use crucially the bounds for $\tau\to1$ that follow from the Fourier expansions~\eqref{eq:asymptexp}.)
\end{proof}
The next lemma is an extension of \cite[Prop. 6.6]{BRS}.

\begin{lemma}\label{lem:Ib}
    For $0<y\le 1/2$, we have
    \[ \int_{-1}^{1} \mathbf{I}(x+iy)^{\alpha} dx \ll_{\alpha} \begin{cases} 1,  & |\alpha|<1, \\ \log \frac{1}{y} , & |\alpha|=1, \\ y^{1-|\alpha|}, & |\alpha|> 1.
    \end{cases} \]
\end{lemma}
\begin{proof}
    We claim that
    \[\mathbf{I}(\tau)\mathbf{I}(\tau+1) \ge 3/4\,,\qquad \tau\in\HH\,,\]
    which reduces the case $\alpha\le0$ to the case $\alpha\ge 0$
    that is covered by \cite[Prop. 6.1]{BRS}.

    To see why the claim holds, let $\mathcal{F}_1:=\{z\colon |z|>1, |\re z|<1/2\}$ denote the standard fundamental domain for $\SL_2(\Z)$ and
    note that $\mathcal{F}'=\mathcal{F}_1\cup T\mathcal{F}_1\cup TS\mathcal{F}_1$
    is a fundamental domain for $\Gamma_{\theta}$ that is translation-equivalent
    to $\mathcal{F}$. Let $\tau_1=\gamma_{\tau}\tau$ and
    $\tau_2=\gamma_{\tau+1}(\tau+1)$ be the representatives of $\tau$ and $\tau+1$
    in $\mathcal{F}$. We need to show that $\im\tau_1\im\tau_2\ge 3/4$.
    By applying $T^2$ if necessary, we may assume that
    $\tau_1,\tau_2$ are in $\mathcal{F}'$.
    Note that for generic $\tau$, we have $\tau_1\ne \tau_2$.
    Since $\tau_1$ and $\tau_2$ are in $\SL_2(\Z)$-equivalent, we have
    $\tau_2=\gamma\tau_1$ with $\gamma$ in $\{T,TS,T^{-1},ST^{-1},TST^{-1}\}$.
    In all cases, the claim follows from the inequality
    $(\im z)^2\ge \frac{3}{4}|z|^2$ which holds for all~$z$ in $\mathcal{F}_1$.
\end{proof}
Our third lemma is \cite[Prop. 6.7]{BRS}.
\begin{lemma}\label{lem:Nupper} We have
    \[ \int_{-1}^1 \mathbf{N}(x+iy)=\frac{2}{\pi^{2}} \log^2 y +O(\log y), \quad y\to 0.\]
\end{lemma}
We are now ready to prove the desired bound for $f_n$. 

\begin{proof}[Proof of \eqref{eq:outsidepw}] In view of \cite[Sec. 7]{BRS}, it suffices to estimate $|f_n(x)| $ for $1\le x\le n$.
    Since
    \begin{equation*} \label{eq:fndef} 
    f_n(x)= \int_{-1+i/n}^{1+i/n} F_g(\tau)e^{-i\pi n \tau }d\tau,
    \end{equation*}
    where $g(z):=e^{i\pi xz}$ and $F_g(\tau)=F(\tau,x)$ (see~\eqref{eq:gfseries}), 
    we get that
    \begin{equation} \label{eq:Fg} | f_n (x)| \le  e^{\pi} \int_{-1}^1  |F_g(t+i/n)| \, dt. \end{equation}
    Since $|g(x+iy)| \le e^{-\pi x y} $, $g$ is a bounded function in $\HH$ and hence of moderate growth. This means that we may employ the machinery of \cite{BRS}. To this end, we find first that
    \begin{equation} \label{eq:Fgupper}  F_g(\tau) \ll x^{-1/2} \max(1,(\im\tau)^{-1}) , \quad \tau\in \mathcal{F} .\end{equation}
    Indeed, by making the same deformation of the path of integration as in the proof of \cite[Prop.~6.1]{BRS} and using Lemma~\ref{lem:ker}, we get
    \[  \left|\int_{J^{-1}(\ell)} K(\tau,z) g(z) dz\right|^2\le    \max(1,(\im \tau)^{-2}) \int_{J^{-1}(\ell)} |g(z)|^2 |dz|. \]
    This bound yields \eqref{eq:Fgupper} because
    \[ \int_{J^{-1}(\ell)} |g(z)|^2 |dz|  \ll x^{-1}.  \]

    To estimate the right-hand side of \eqref{eq:Fg}, we start from \cite[Eq. (6.3), p. 40]{BRS}: 
    \begin{equation} \label{eq:start} |F_g(\tau)| (\im \tau)^{1/4} \le |F_g(\tau_0)| (\im \tau_0)^{1/4}
        +|\psi(\tau_0)| (\im \tau_0)^{1/4} + \sum_{i=1}^{\mathbf N(\tau)} |\psi(\tau_i)| (\im \tau_i)^{1/4},\end{equation}
    where
    \[ \psi(z)\coloneqq g(z) +(z/i)^{-1/2}g(-1/z),\]
    the point $\tau_0$ lies in $\cup_{j\in\Z}(2j+\bF)$, and 
    \begin{equation} \label{eq:chain} \tau_j\coloneqq T^{2m_j}S\cdots T^{2m_1}S\tau_0, \quad \tau=S\tau_{\bN(\tau)}.\end{equation}
    In view of \eqref{eq:Fgupper} and Lemma~\ref{lem:Ib}, the first term on the right-hand side of \eqref{eq:start} yields a term of order $n^{1/4}x^{-1/2}$.
    To deal with the two remaining terms, we notice that we now have (writing $\eta_i:=\im \tau_i$)
    \[ |\psi(\tau_i)|  \ll e^{-\pi x \eta_i} + |\tau_i |^{-1/2} e^{-\pi x  \eta_{i+1}}  
         =
        e^{-\pi x  \eta_i} + \eta_i^{-1/4}\eta_{i+1}^{1/4} e^{- \pi x \eta_{i+1}} .\]
    This gives
    \begin{equation} \label{eq:etai}
         |\psi(\tau_0)| \eta_0^{1/4} \ll \eta_0^{1/4} e^{- \pi x \eta_0} + \eta_{1}^{1/4} e^{- \pi x \eta_{1}} \ll x^{-1/4} \end{equation}
    as well as
    \begin{equation}\label{eq:alt1} \sum_{i=1}^{\mathbf N(\tau)} |\psi(\tau_i)| \eta_i^{1/4} \le 2 \mathbf N(\tau) \max_{1\le i\le \mathbf N(\tau)+1} \eta_i^{1/4} e^{-\pi x \eta_{i}} \ll \mathbf N(\tau) x^{-1/4} .\end{equation}
    %since $\eta_i\ge n^{-1}$ for all $i$ and $T\le n$.
    Using \eqref{eq:etai} and \eqref{eq:alt1} to estimate the two latter terms on the right-hand side of \eqref{eq:start} and also invoking Lemma~\ref{lem:Nupper}, we get the desired bound.
\end{proof}
\subsection{The integral of sums of squares by a Carleson measure argument}
\label{sec:sumsquares}
We will now prove the upper bound in \eqref{eq:sqrsum} of Theorem~\ref{thm:sqrsum}. The technique to be developed here is an elaboration of that used in the preceding section. To ease the exposition, we split the argument into three subsections.

\subsubsection{The space \texorpdfstring{$H^2(\HH)$}{H^2(H)}} We begin with some general facts about functions $f$ in $H^2(\HH)$ and the associated function $F_f$ which is defined as
    \begin{equation} \label{eq:defF}
    F_f(\tau) := \frac{1}{2} \int_{-1}^1 K(\tau,z) f(z)dz
    \end{equation}
for $\tau$ in the fundamental domain $\mathcal F$.
We recall that $f$ is in $H^2(\HH)$ if $f$ is analytic in $\HH$ and
    \[ \| f \|_{H^2(\HH)}^2\coloneqq \sup_{y>0} \int_{-\infty}^{\infty} |f(x+iy)|^2 dx <\infty. \]
A function $f$ in $H^2(\HH)$ has nontangential limits at almost every point of the real line. The corresponding limit function, which we also denote by $f$, has $L^2$ norm $\| f\|_{H^2(\HH)}$. By this correspondence, we may think of $H^{2}(\HH)$ as the subspace of $L^2(\mathbb R)$ consisting of functions whose Fourier transform vanishes for negative frequencies. See \cite[Ch. II]{Ga} for general information about $H^p$ spaces.

Functions in $H^2(\HH)$ are of moderate growth. More precisely, we have 
\begin{equation} \label{eq:pointbound} |f (u+iv)|\le  v^{-1/2} \|f\|_{H^2(\HH)} \end{equation}
which follows from the Cauchy formula
    \[ f(w)=\frac{1}{2\pi i} \int_{-\infty}^{\infty} \frac{f(x)}{x-w} dx.  \]
    The corresponding bound for $F_f$ is as follows.
\begin{lemma} \label{lem:CauchyF}
    There exists an absolute constant $C$ such that if $f$ is in $H^2(\HH)$ and $w$ is in~$\bF$, then
    \[ |F_f (w)| \le C \max(1, v^{-1}) \| f \|_{H^2(\HH)} .\]
\end{lemma}

\begin{proof}
    By the definition \eqref{eq:defF} and Lemma~\ref{lem:ker}, we have
    \[ |F_f(u+iv)|^2 \le C \max(1,v^{-2}) \int_{\ell} |f(z)|^2 |dz| . \]
    It is clear that arclength measure on $\ell$ constitutes a Carleson measure (see \cite[p. 30]{Ga}), so we get
    \[ |F_f(u+iv)|^2 \le C' \max(1,v^{-2}) \|f\|_{H^2(\HH)}^2  \]
    by Carleson's embedding theorem \cite[p. 61]{Ga}. 
\end{proof}

\subsubsection{Estimation by means of a Carleson measure condition} We fix a number $\tau=t+i/\xi$, where $\xi\ge 10$.  The chain $\tau_0, \ldots, \tau_{\mathbf{N}(\tau)}$ is as before, and our goal is to estimate 
\begin{equation} \label{eq:start2}
    |F_f(\tau)| (\im \tau)^{1/4} \le |F_f(\tau_{0})| (\im \tau_{0})^{1/4}
    + \mathcal{C}_f(\tau),\end{equation}
    where $\mathcal{C}_f(\tau)$ will be the modulus of a sum similar to that on the right-hand side of \eqref{eq:start}. 
We denote by $\mathbf C (\tau)$ the smallest constant $C$ such that
\begin{equation} \label{eq:csum}
    \mathcal{C}_f(\tau) \le C \| f \|_{H^2(\HH)}
\end{equation}
holds for all $f$ in $H^2(\HH)$. Our key estimate is contained in the following lemma.
\begin{lemma} \label{lem:Carl} We have
    \begin{equation} \label{eq:Carleson} \int_{-1}^{1} \mathbf{C}^2 (t+i/\xi) dt
        \ll  \xi^{1/2} \log^2 \xi .\end{equation}
\end{lemma}

We prepare for the proof of this lemma by identifying a suitable upper bound for $\mathcal{C}_f(\tau)$. We begin by defining $n_0$ as the largest $\ell$ such that either
    \[ |\tau_\ell-1|^2 \le \eta_\ell \quad \text{or} \quad |\tau_\ell+1|^2 \le \eta_\ell.\]
Without loss of generality, we assume that $|\tau_{n_0}-1|^2\le \eta_0$. Then all
the points $\tau_0, \ldots, \tau_{n_0}$ lie on the circle 
\[ D(\tau_{n_0})\coloneqq \Big\{z:\ \Big|z-1-i\frac{|\tau_{n_0}-1|^2}{2\eta_{n_0}}\Big|=\frac{|\tau_{n_0}-1|^2}{2\eta_{n_0}}\Big\} ,\]
and hence
\begin{equation} \label{eq:Iestimate}
     \frac{|\tau_{n_0}-1|^2}{2\eta_{n_0}} \le \mathbf{I}(\tau) \le \frac{|\tau_{n_0}-1|^2}{\eta_{n_0}}. \end{equation}
We may express $C_f$ as
\[ C_f(\tau)=\Big|\sum_{i=0}^{\mathbf N(\tau)} \lambda_i \psi(\tau_i) (\im \tau_i)^{1/4}\Big|,\]
where $\lambda_i$ are unimodular numbers and
\begin{equation} \label{eq:psif} \psi(z)\coloneqq f(z)+(z/i)^{-1/2} f(-1/z). \end{equation} 
We may compute $\lambda_i$ in the range $0\le i\le n_0$. Indeed, we find then that
\[ \mathcal{C}_f(\tau)\le \mathcal{C}_{f,1}(\tau)+\mathcal{C}_{f,2}(\tau), \]
where
\begin{equation}\label{eq:Cf1}  \mathcal{C}_{f,1}(\tau)\coloneqq \eta_{n_0}^{1/4}\, \Big|\sum_{j=0}^{n_0}
\psi(\tau_{n_0-j})(-1)^j\Big(\frac{j(\tau_{n_{0}}-1)+1}{i}\Big)^{-1/2}\Big| , \ \tau_{n_0-j}=\frac{(j+1)\tau_{n_0}-j}{j(\tau_{n_0}-1)+1},\end{equation}
and
\[\mathcal{C}_{f,2}(\tau)\coloneqq \sum_{j=n_0+1}^{\mathbf{N}(\tau)} |\psi(\tau_j)| (\im \tau_j)^{1/4}. \]
We decompose $\mathcal{C}_f(\tau)$ in this way to take advantage of the cancellations in the sum on the right-hand side of \eqref{eq:Cf1}. To this end, we have the following result.

\begin{lemma} \label{lem:cancel} There exist indices $0=j_1<\cdots < j_k=n_0$ along with points $\kappa_{j_{1}},\ldots, \kappa_{j_{k}}$ such that 
\begin{align} \label{eq:separate} |\tau_{j_{\ell}}-\kappa_{j_\ell}|& \le \eta_{j_\ell}/2, \quad
\#\{\kappa_{j_i}: \ |\kappa_{j_i}-\kappa_{j_\ell}|\le \im \kappa_{j_{\ell}}\}  \ll 1, \quad 1\le \ell \le k,  \\ \label{eq:sumkappa}
\mathcal{C}_{f,1}(\tau) & \ll \sum_{\ell=1}^k |\psi(\kappa_{j_\ell})| (\im \kappa_{j_\ell})^{1/4}\big(1+\mathbf{I}(\tau)^{-1/2}(\im{\kappa_{j_\ell}})^{1/2}\big).
\end{align} 
\end{lemma}
\begin{proof}
We set $j_0=0$ and define the numbers $j_\ell$ inductively by 
requiring 
\[ |\tau_{j_\ell}-\tau_{j_{\ell+1}}| \le \eta_{j_\ell}/4 < |\tau_{j_\ell}-\tau_{j_{\ell+1}+1}| \] 
as long as these inequalities hold for some $j_{\ell+1}<n_0$, and otherwise we terminate the process by setting $j_{\ell+1}=n_0$. A computation shows that
\[ |\tau_{n_0-j}-\tau_{n_0-(j+1)}|\asymp \frac{|\tau_{n_0}-1|^2}{\eta_{n_0}} \,\eta_{n_0-j},\]
whence 
\begin{equation} \label{eq:length} j_{\ell+1}-j_{\ell} \ll \mathbf{I}(\tau)^{-1}\end{equation}
by \eqref{eq:Iestimate}. 

To estimate $\mathcal{C}_{f,1}$, we write
\[ \mathcal{C}_{f,1}\le \sum_{\ell=1}^k \sigma_\ell,\]
where 
\begin{equation} \label{eq:sigmai}  \sigma_\ell\coloneqq \eta_{n_0}^{1/4}\, \Big|\sum_{j=j_\ell}^{j_{\ell+1}-1}
\psi(\tau_{n_0-j})(-1)^j\Big(\frac{j(\tau_{n_{0}}-1)+1}{i}\Big)^{-1/2}\Big|.\end{equation}
We will estimate $\sigma_\ell$ by partial summation and Cauchy estimates. To begin with, setting $H(x)\coloneqq \sum_{M\le j\le x} (-1)^j$, we obtain by partial summation
\begin{align*} \sum_{j=M}^{M+N}(-1)^j(j(\tau_{n_{0}}-1)+1)^{-1/2}
=& H(M+N)((M+N)(\tau_{n_0}-1)+1)^{-\frac12} \\ &-\frac12(\tau_{n_0}-1)\int_M^{M+N} H(x)(x(\tau_{n_0}-1)+1)^{-\frac32}dx.  \end{align*}
If $M=j_\ell$ and $M+N\le j_{\ell+1}-1$, then using \eqref{eq:length} and the bound  $|H(x)|\le 2$, we get 
\begin{equation} \label{eq:carlbox} \eta_{n_0}^{1/4} \sum_{j=j_\ell}^{j_\ell+N}(-1)^j(j(\tau_{n_{0}}-1)+1)^{-1/2}
\ll \eta_{n_0-j_\ell}^{1/4}+\mathbf{I}(\tau)^{-1/2} \eta_{n_0-j_\ell}^{3/4}.\end{equation}
By summation by parts and \eqref{eq:carlbox}, 
%\[ \psi(z)\coloneqq f(z) +(z/i)^{-1/2}f(-1/z)\]
\[ \sigma_{\ell} \ll \big(\mathbf{I}(\tau)^{-1/2} \eta_{n_0-j_\ell}^{3/4}+\eta_{n_0-j_\ell}^{1/4}\big) \Big( |\psi(\tau_{n_0-j_\ell})|+\int_{\tau_{n_0-j_{\ell}}}^{\tau_{n_0-j_{\ell+1}+1}}
|\psi'(z)| |dz|\Big),\]
where we on the right-hand side integrate along the line segment between the two points. Recalling \eqref{eq:psif} and \eqref{eq:pointbound}, we obtain the desired result by making a Cauchy estimate for a circle of radius $\eta_{n_0-j_\ell}/4$ about each $z$ in that line segment. Specifically, we obtain
$\kappa_{j_\ell}$ as the $w$ at which $|\psi(w)|$ takes its maximum in the set of points at distance at most $\eta_{n_0-j_\ell}/4$ from the segment between $\tau_{n_0-j_\ell}$ and $\tau_{n_0-j_{\ell+1}+1}$. 
\end{proof}    

We now introduce the notation
    \[ \Gamma_{\alpha}(u):=\{x+iy \in \HH : |x-u|<\alpha y\}, \]
    where $\alpha>0$. We require the following description of $L^1$ Carleson measures for $H^2(\HH)$.
\begin{lemma}\label{lem:carllueck}  Let $\mu$ be a nonnegative
    measure on $\HH$. Then
    \begin{equation} \label{eq:carllueck} \left( \int_{\HH} |G(z)| d\mu(z)\right)^2 \ll_{\alpha}
    \int_{-\infty}^{\infty} \left(\int_{\Gamma_{\alpha}(u)} y^{-1} d\mu(x+iy)\right)^2 du
    \  \| G \|_{H^2(\HH)}^2. \end{equation} \end{lemma}
This result can be found in \cite[Thm. C]{Lu}. As shown there, the inequality is an immediate consequence of Hardy and Littlewood's theorem on the $L^p$ boundedness of the nontangential maximal function \cite[p. 55]{Ga}. A more subtle result, also proved in \cite{Lu}, is that the order of magnitude of the constant
    \[ \int_{-\infty}^{\infty} \left(\int_{\Gamma_{\alpha}(u)} y^{-1} d\mu(x+iy)\right)^2 du \]
    on the right-hand side of \eqref{eq:carllueck} is optimal.
    
  \begin{proof}[Proof of Lemma~\ref{lem:Carl}] We let $\mathbf{C}_i(\tau)$, $i=1,2$, be the smallest constant $C$ such that 
  \[ \mathcal{C}_{f,i}(\tau)\le C \| f \|_{H^2(\HH)} \]
  holds for all $f$ in $H^2(\HH)$. It clearly suffices to prove that 
  \begin{equation} \label{eq:C12} \int_{-1}^{1} \mathbf{C}_i^2 (t+i/\xi) dt
        \ll  \xi^{1/2} \log^2 \xi , \quad i=1,2.\end{equation}
 The proof in either case is essentially the same. 
 
To prove \eqref{eq:C12} when $i=2$, we set
    \[ \mu=\sum_{j=n_0+1}^{\mathbf N(\tau)-1} \Big(
    \eta_j^{1/4} \delta_{\tau_j}+ \eta_{j+1}^{1/4}\delta_{S\tau_j} \Big) \] and $\alpha=1/4$. By Lemma~\ref{lem:carllueck}, we get
    \begin{align} \nonumber \int_{-1}^1 \mathbf C_{2}^2 (t+i/\xi) dt & \ll \int_{-1}^1
        \int_{-\infty}^\infty \Big(\sum_{\tau_i\in \Gamma_{1/4} (u)}
        \eta_i^{-3/4}+\sum_{S\tau_i\in \Gamma_{1/4}(u)} \eta_{i+1}^{-3/4}\Big)^2 du\,  dt \\
        & \label{eq:double} \le 2 \int_{-1}^1 \int_{-\infty}^\infty \Big( \sum_{\tau_i\in
            \Gamma_{1/4}(u)} \eta_i^{-3/4}\Big)^2du\,  dt+2\int_{-1}^1
        \int_{-\infty}^\infty\Big(\sum_{S\tau_i\in \Gamma_{1/4}(u)} \eta_{i+1}^{-3/4}\Big)^2
        du\,  dt. \end{align}
    We are thus left with the problem of estimating the two double integrals in
    \eqref{eq:double}. Both can be dealt with by the same argument, and we therefore give
    the details only for the first of them. Opening up the square, we find that
    \begin{align*} \int_{-1}^1 \int_{-\infty}^\infty \Big( \sum_{\tau_i\in \Gamma_{1/4}
            (u)} \eta_i^{-3/4}\Big)^2du\,  dt & \ll  \int_{-1}^1 \sum_{i=n_0+1}^{\mathbf N(t+i/\xi)}
        \eta_i^{-3/4} \sum_{j\ge i,\, |\re(\tau_i-\tau_j)|<\eta_i/2} \eta_j^{1/4}
        dt  .\end{align*} If we are able to prove that
    \begin{equation} \label{eq:goal}
        \sum_{j\ge i,\, |\re(\tau_i-\tau_j)|<\eta_i/2} \eta_j^{1/4}
        \ll \eta_i^{1/4} ,
    \end{equation}
    then the desired bound follows, because we get
    \[ \int_{-1}^1 \int_{-\infty}^\infty \Big( \sum_{\tau_i\in \Gamma_{1/4}(u)}
    \eta_i^{-3/4}\Big)^2du\,  dt  \ll  \int_{-1}^1 \sum_{i=0}^{\mathbf N(t+i/\xi)}
    \eta_i^{-1/2}  dt
    \ll  \xi^{1/2}  \int_{-1}^1 \mathbf N(t+i/\xi) dt\]
    and we may conclude by an application of Lemma~\ref{lem:Nupper}.

    It remains to prove \eqref{eq:goal}. %It is convenient to note that the %$\tau_i$
    %with $\eta_i$ large, say $\eta_i\ge 1/10$, will yield a contribution to
    %\eqref{eq:csum} that can be estimated trivially, so we may assume that $\eta_i\le 1/10$. 
    We set
    \[ Q_i:=\{u+iv: \ |u-\operatorname{Re} \tau_i |\le \eta_i/2, \ 0<v\le \eta_i \} \]
    and assume that $\tau_i$ lies in the rectangle $\{u+iv: \ 1<u<4/3, \ 0<v<1/9\},$ 
    where the bound $0<v<1/9$ comes from the requirement that $(\xi_i-1)^2>\eta_i-\eta_i^2$. It
    suffices to consider this situation, because the problem is trivial if $\tau_i$ lies to the right of this rectangle, and we may argue in exactly the same way if $\xi_i<-1$. 

    Assume that $\tau_j, \tau_{j+1},\ldots, \tau_{j+k}$ are all in $Q_i$ but that $\tau_{j+k+1}$ is not. We intend to prove that this implies that $k\le 5$ and also that
    $\eta_{\ell}<9\eta_{j+k}/25$ if $\ell>j+k+1$ and $\tau_{\ell}$ is in $Q_i$. It is clear that this will yield \eqref{eq:goal}.

    Our assumption about $\tau_j, \tau_{j+1},\ldots, \tau_{j+k}$ implies that
    \[ \tau_{j+\nu+1} = T^2 S \tau_{j+\nu}, \quad 0\le \nu\le k-1. \]
    In general, for a point $w=u+iv$ in $\HH$, we have
    \begin{equation} \label{eq:small} u- \operatorname{Re} (T^2 S
        w)=u-2+\frac{u}{|w|^2}=\frac{u(1-u)^2+uv^2-2v^2}{u^2+v^2} .
        %> \frac{uv-2v^2}{u^2+v^2}. 
        \end{equation}
        If $w$ is in $Q_i$, then we get 
        \[ u- \operatorname{Re} (T^2 S
        w)\ge \frac{9}{16}\big((1-\re\tau_i)^2-\eta_i(\re\tau_i-1)-5\eta_i^2/4\big). \]
    Also taking into account that $(1-\re\tau_i)^2>\eta_i-\eta_i^2$, we infer from this bound that
    \[ u- \operatorname{Re} (T^2 S
        w)>\frac{9}{16}\big((2-\re\tau_i) \eta_i -5\eta_i^2/2\big)\ge \frac{7}{32}\eta_i. \]
    This implies that $k\le 5$.
    
    Finally, it is clear that if $\tau_{\ell}$ is in $Q_i$ for some $\ell>j+k+1$, then there must exist an $\ell'$ between $j+k$ and $\ell$ such that
    $|\operatorname{Re}\tau_{\ell'}|\ge 5/3$, and this implies that $\eta_{\ell}<9\eta_{j+k}/25$.

We now turn to the proof of \eqref{eq:C12} for $i=1$. Using \eqref{eq:separate} of Lemma~\ref{lem:cancel}, we find that
\[ \textbf{I}(\tau)^{-1/2} \sum_{i=1}^k |\psi(\kappa_{j_i})| (\im{\kappa_{j_i}})^{3/4}\ll \mathbf{I}(\tau)^{-1/2} \xi^{1/4} \| f\|_{H^2(\HH)}, \] 
and so the corresponding part of the sum in \eqref{eq:sumkappa} contributes a term $\ll \xi^{1/2} \log \xi$ to the integral in \eqref{eq:C12} by the Cauchy--Schwarz inequality and Lemma~\ref{lem:Ib}. 

To deal with the remaining part of the sum in \eqref{eq:sumkappa}, we note first that the terms corresponding to points $\kappa_{j_i}$ lying above say the horizontal line $y=r/10$, with $r$ the radius of $D(\tau)$, add up to at most $O(\xi^{1/4}) \|f\|_{H^2(\HH)}$. Hence we may remove these terms from \eqref{eq:sumkappa} and then, thanks to Lemma~\ref{lem:cancel}, repeat essentially word for word the proof for the case $i=2$. We omit the details of this argument.
\end{proof}

\subsubsection{Proof of the upper bound in \texorpdfstring{\eqref{eq:sqrsum}}{1.7} of Theorem \texorpdfstring{\ref{thm:sqrsum}}{1.3}} We start from the formula
    \begin{equation} \label{eq:l2sum} e^{-\pi} \sum_{n\le \xi} |f_n(x)|^2\le
        \frac12\int_{-1}^{1} |F_g(t+i/\xi)|^2 dt ,\end{equation}
where as before $g(z):=e^{i\pi x z}$. %This bound is only interesting when $\xi\ge x$
    %which will be assumed in the sequel.
By duality, we have then
    \begin{equation} \label{eq:duality} \int_0^\infty \sum_{n\le \xi} |f_n(x)|^2 dx\le \frac{e^{\pi}}{2} \sup_{\| h\|_2=1} \left| \int_{0}^{\infty}
    \int_{-1}^{1} F_g(t+i/\xi) h(t,x) \, dt dx \right| ,\end{equation}
where
    \[ \| h\|_2^2\coloneqq \int_{0}^{\infty} \int_{-1}^1 |h(t,x)|^2 dt dx. \]
By
Fubini's theorem, the function
    \[ H_t(z) := \int_{0}^{\infty} h(t,x) e^{\pi i x z } dx  \]
will be in the Hardy space $H^2(\HH)$ for almost all $t$ when $\| h \|_2=1$, and we have
    \begin{equation}\label{eq:Planch}
    \int_{-1}^1 \|H_t\|_{H^2(\HH)}^2 dt =2 \end{equation}
by the Plancherel identity. The double integral on the right-hand side of \eqref{eq:duality} can be written as
    \begin{equation}\label{eq:intH}
    \int_{0}^{\infty} \int_{-1}^{1} F_g(t+i/\xi) h(t,x) dt dx = \int_{-1}^1
    F_{H_t}(t+i/\xi) dt.
    \end{equation}
Invoking~\eqref{eq:start2} with $f=H_{t}$ and using Lemma~\ref{lem:CauchyF}, we find that
    \[ |F_{H_t}(t+i/\xi)| \xi^{-1/4} \ll  \max(\mathbf{I}(\tau)^{-3/4},\mathbf{I}(\tau)^{1/4})\,\|H_t \|_{H^2(\HH)}+  \mathbf C(t+i/\xi)\, \| H_t \|_{H^2(\HH)}.\]
Integrating and applying the Cauchy--Schwarz inequality along with Lemma~\ref{lem:Ib} and Lemma~\ref{lem:Carl}, we find that
    \[ \int_{-1}^1 |F_{H_t}(t+i/\xi)|  dt  \ll \xi^{1/2} \log \xi \, \Big(
    \int_{-1}^{1} \| H_t \|_{H^2(\HH)}^2 \Big)^{1/2} .\]
 Now recalling \eqref{eq:intH} and \eqref{eq:Planch} and returning to \eqref{eq:duality}, we get the desired bound
    \[ \int_0^\infty \sum_{n\le \xi} |f_n(x)|^2 dx \ll \xi \log^2 \xi.  \]

\section{Further estimates of \texorpdfstring{$f_n$}{f\_n} and extraneous zeros close to 0} \label{sec:globalzeros}
%{Asymptotic expression for \texorpdfstring{$f_n(x)$}{f\_n(x)} for %\texorpdfstring{$x<0$}{x<0}}
In this section, we will prove~\eqref{eq:cartw} and~\eqref{eq:tomininf} of Theorem~\ref{thm:global} (note that~\eqref{eq:negint} directly follows from~\eqref{eq:aapprox}). We will then use the former bound along with \eqref{eq:negint} to establish (iv) of Theorem~\ref{thm:zeros}. We begin  with our asymptotic estimate \eqref{eq:tomininf} for negative arguments of~$f_n$.

\subsection{Estimate of \texorpdfstring{$f_n(x)$}{f\_n(x)} for \texorpdfstring{$x<0$}{x<0}}
\label{sec:fnasymp}
Our aim is to prove~\eqref{eq:tomininf} of Theorem~\ref{thm:global}, namely,
\begin{equation}
       \label{eq:asymp1}
        f_n(-x) \sim \frac{e^{2\pi \sqrt{nx}}}{2\sqrt{n}}\,,
        \qquad x \to +\infty\,.
  \end{equation}
First, let us record a few simple uniform estimates for $g_n$ in the fundamental domain. From~\eqref{eq:aapprox} it follows that
    \begin{equation*} %\label{eq:gnglobalbound}
    |g_n(\tau)| \le  e^{\pi n y}(1+O(e^{\pi n(1-y)}))\,,\qquad \tau=x+iy,\quad y\ge1,
    \end{equation*}
and from~\eqref{eq:atildeapprox} one gets a rough estimate $|g_n(\tau)|\ll e^{2\pi n}$ for $\tau=x+iy$ in $\bF$ with $y\le 1$. Thus in the fundamental domain, we have 
    \begin{equation} \label{eq:gnglobalbound2}
    |g_n(\tau)|\ll e^{\pi n\max(2,y)}\,,\qquad \tau=x+iy.
    \end{equation}
% In Section~\ref{sec:globalestimate} we will need the following bound (that also follows from~\eqref{eq:atildeapprox}):
%    \begin{equation} \label{eq:gnglobalbound3}
%     |g_n(\tau)|\ll e^{2\pi n y}\,,\qquad \tau=x+iy,\quad 0<y\le1.
%     \end{equation}
    
We will apply a version of the circle method. We fix $n$, and define $\eps \coloneqq \sqrt{n/x}$. We will assume that $x\ge 10n$. Then
    \[f_n(-x) =
    \frac{e^{\pi x\eps}}{2}\int_{-1}^{1}g_n(t+i\eps)e^{-\pi i xt} dt
    +\sin(\pi x)\int_{0}^{\eps}g_n(1+it)e^{\pi xt} dt\,.\]
Note that the second integral is $\ll e^{2\pi n+\pi \sqrt{nx}}$, so it will be negligible compared to $e^{2\pi \sqrt{nx}}$. Thus it is enough to estimate the first integral. For this we split the interval $[-1,1]=I_1\sqcup I_2$, where
    \[I_1\coloneqq \Big\{t\in[-1,1] \colon \bH(t+i\eps)\le
    \frac{1}{2\eps}\Big\}\,.\]
From~\eqref{eq:gnglobalbound2} we get $|\eps^{3/4}g_n(z)|=|\im (\tau_z)^{3/4}g_n(\tau_z)| \ll \bH(t+i\eps)^{3/4}e^{\pi n\max(2,\bH(t+i\eps))}$, and hence
    \[\Big|\frac{e^{\pi x\eps}}{2}\int_{I_1}g_n(t+i\eps)e^{-\pi i xt} dt\Big|
    \ll e^{\frac{3}{2}\pi \sqrt{nx}}\,.\]
From the proof of~\cite[Proposition~6.6]{BRS}, it is easy to see that $I_2=(-\eps,\eps)$. Thus, it is enough to show that, for $x\to\infty$,
    \[A = \int_{-\eps}^{\eps}g_n(t+i\eps)e^{-\pi i x(t+i\eps)} dt
    \sim \frac{e^{2\pi \sqrt{nx}}}{\sqrt{n}}\,.\]
After applying the transformation $z\mapsto -1/z$ to $g_n$ we get
    \[A = \int_{-h+ih}^{h+ih}g_n(z)(z/i)^{-1/2}e^{\pi i x/z} dz\,, \]
where $h=\tfrac{\eps^{-1}}{2}$. Since $h\to \infty$ as $x\to \infty$,
we have $g_n(z)=e^{-\pi i nz}+o(1)$ for $z=t+ih$, so the problem is reduced to showing that
    \[\int_{-h+ih}^{h+ih}	(z/i)^{-1/2}e^{-\pi i nz+\pi i x/z} dz \sim
    \frac{e^{2\pi \sqrt{nx}}}{\sqrt{n}}\,,\]
which after rescaling $z$ by $2h$ reduces to the study of the asymptotics of
    \[\int_{\frac{-1+i}{2}}^{\frac{1+i}{2}}(z/i)^{-1/2}e^{T(z/i+i/z)} dz\,,\qquad T\to \infty\,.\]
This can be easily proved using the saddle point method if we choose the
contour	of integration to be $\gamma_1\cup \gamma_2\cup \gamma_3$ going
through	$\frac{-1+i}{2}\to \frac{-1+i\sqrt{3}}{2} \to \frac{1+i\sqrt{3}}{2}\to \frac{1+i}{2}$, where $\gamma_1$ and $\gamma_3$
are straight line segments and $\gamma_2$ is an arc of the unit circle
centered at $0$. Then the integral over $\gamma_1$ and $\gamma_3$ is trivially bounded by $Ce^{T\sqrt{3}}$, while for $\gamma_2$ we calculate
    \[\int_{\gamma_2}(z/i)^{-1/2}e^{T(z/i+i/z)} dz 
    = \int_{-\pi/6}^{\pi/6}\cos(t/2)e^{2T\cos(t)} dt
    \sim e^{2T}\sqrt{\pi/T}\,,\quad T\to\infty,\]
by a simple application of the Laplace method. This proves~\eqref{eq:asymp1}.

\subsection{A global estimate for \texorpdfstring{$f_n$}{f\_n}}
\label{sec:globalestimate}
We will now extend the above proof technique to establish \eqref{eq:cartw} of Theorem~\ref{thm:global}, i.e., the bound 
\begin{equation} \label{eq:global} f_n(x+iy) \ll  n^{\frac14} (1+\log^2n) e^{\pi |y|} e^{2\pi \sqrt{n} \operatorname{Re} \sqrt{-z}},\end{equation}
which holds    uniformly in $z$ and $n$.
    
We know by Lemma~\ref{prop:fnapprox} that for $\re \sqrt{z}>2\sqrt{n}$
    \[|f_n(z)| \ll |\sin(\pi z)| \le e^{\pi|y|}.\]
Let us now assume that $x\le -n^2$ and $|y|\le 4\sqrt{n|x|}$. 
We start again from the representation
    \begin{equation} \label{eq:fn} f_n(z) =
    \frac{e^{-\pi z\eps}}{2}\int_{-1}^{1}g_n(t+i\eps)e^{\pi i zt} dt
    -\sin(\pi z)\int_{0}^{\eps}g_n(1+it)e^{-\pi zt} dt\,.\end{equation}
where we set $\eps=\sqrt{n/|x|}$. Trivially, the second term on the right-hand side of \eqref{eq:fn} satisfies
    \begin{equation*} %\label{eq:sec} 
    \sin(\pi z)\int_{0}^{\eps}g_n(1+it)e^{-\pi zt} dt \ll e^{\pi |y|} e^{\pi(2n+\varepsilon|x|)}.
    \end{equation*}
which is bounded by $e^{\pi|y|+3/2\pi \sqrt{n|x|}}$. Following the rest of the argument from Section~\ref{sec:fnasymp} with trivial modifications, we obtain the bound
    \[|f_n(\tau)| \ll e^{\pi |y|}e^{2\pi \sqrt{n|x|}}\,, \qquad x\le -n^2,\quad  |y|\le 4\sqrt{n|x|}\,.\]
Collecting the above estimates (noting that for $\re w>0$ we have $\re\sqrt{w}\ge \sqrt{\re w}$), we arrive at the bound
    \begin{equation} \label{eq:uniform} f_n(z) \ll e^{\pi |y|} e^{2\pi \sqrt{n} \operatorname{Re} \sqrt{-z}}, \end{equation}
which holds uniformly when $\re \sqrt{z} \ge 2\sqrt{n}$ or $x \le -n^2$. Here $\sqrt{-z}$ is defined in the slit plane $\mathbb{C} \setminus [0,\infty)$ and is positive when $z$ is negative. This bound makes sense for all $z$ since $\operatorname{Re} \sqrt{-z}$ becomes zero when $z$ tends to a positive number. 
%By combining \eqref{eq:uniform} with the bound
%\begin{equation} \label{eq:positive} f'_n(x) \ll n^{5/4} \log^2 n ,\end{equation}

Finally, to establish \eqref{eq:global}, we set
    \[ \Omega(n):=\left\{z: \ \re \sqrt{z} < 2\sqrt{n},\; \re z>-n^2\right\} \smallsetminus [0,4n) \]
and consider the function
    \[ G(z):=f_n(z) e^{-2\pi \sqrt{n} \sqrt{-z} } . \]
By \eqref{eq:negint} and the fact that $f_n(m)=0$ for positive integers $m\neq n$, we have control of the behavior of this function at the integers. In particular, the function 
    \[ G_0(z):=z \sum_{k\in \mathbb Z \setminus \{0\}} \frac{G(k)}{k} \frac{\sin \pi (z-k)}{\pi (z-k)} . \]
is well defined and satisfies $G_0(z)\ll e^{\pi |y|}$. Then the function $(G(z)-G_0(z))/\sin \pi z $ is $\ll n^{\frac14} \log^2 n$ on the boundary of $\Omega(n)$. Hence, by the maximum modulus principle, it has the same bound in $\Omega(n)$. Since $G_0(z)\ll e^{\pi|y|}$, the desired result \eqref{eq:global} follows.
%We would have obtained the following slightly better result for $z$ in $\Omega(n)$ if we multiplied $G$ and $H$ by $(1+z)^{1/4}$:
%\begin{equation} \label{eq:better}  f_n(x+iy) \ll  n^{1/4} (1+|z|)^{-1/4} (\log^2 n) e^{\pi |y|} e^{2\pi \sqrt{n} \operatorname{Re} \sqrt{-z}}.\end{equation}
%This follows from Theorem 3 in our new paper.

\subsection{Extraneous zeros close to \texorpdfstring{$0$}{0} in the right half-plane}
We will now use the above global estimate along with \eqref{eq:negint} to  prove (iv) of Theorem~\ref{thm:zeros}. We begin by applying Jensen's formula to the circle of radius $R\ge 2$ centered at $-1$. By \eqref{eq:negint} and \eqref{eq:cartw} of Theorem~\ref{thm:global}, we then get
    \[ \int_0^R \frac{n(t)}{t} dt \le  \log n + 2R + 4\sqrt{nR} + C \]
for an absolute constant $C$, where $n(t)$ is the number of zeros of $f_n$ in the disc $|z+1|\le t$. This gives the upper bound in (i).

The lower bound requires more work. We will get it by counting zeros in the disc $|z|\le Mr$ for a suitable $M>1$ in two different ways. We will assume that $r$ is a positive integer. We set
    \[ \varphi_-(z):=- M^2r \frac{r+z}{M^2r+z} \]
and notice that $\varphi_-$ is a self-map of $|z|\le Mr$ and that $\varphi_-(0)=-r$. Let $n_-(t)$ be the counting function for the zeros $z_j$ of $f_n(\varphi_-(z))$ in the disc $|z|\le t$, where we only count zeros $z_j$ with $\re \varphi_-(z_j) \le 0$.
By Jensen's formula and \eqref{eq:negint}, we find that we have arrived at 
    \begin{equation} \label{eq:Jensen} 2\pi \sqrt{nr}-\frac{1}{2}\log n +C+\int_0^{Mr} \frac{n_-(t)}{t} dt \le
        % \log |f_n(-r)| \le
        \frac{1}{2\pi} \int_0^{2\pi} \log |f_n(Mre^{i\theta})| \big|\varphi_-'(Mre^{i\theta})\big| d\theta \end{equation}
for an absolute constant $C$. Since
    \[ \frac{M-1}{M+1}\le \left|\varphi_-'(Mre^{i\theta})\right|\le \frac{M+1}{M-1},\]
we see that
    \begin{align*}  \int_0^{2\pi} \log |f_n(Mre^{i\theta})|& \big|\varphi_-'(Mre^{i\theta})\big| d\theta \le \int_0^{2\pi} \log |f_n(Mre^{i\theta})| d\theta  \\
        & +  \frac{2}{M-1} \int_0^{2\pi} |\log f_n(Mre^{i\theta})| d\theta .\end{align*}
We now apply the subharmonic mean value inequality to
    \[ z \mapsto \log \Big|f_n\big(-M^2r^2\frac{1+z}{M^2r^2+z}\big)\Big| \] 
for the circle $|z|=Mr$. Invoking again \eqref{eq:cartw} and \eqref{eq:negint} for $m=1$, we then get
    \[ \int_0^{2\pi} |\log f_n(Mre^{i\theta})| d\theta \le C \sqrt{Mr} \]
for an absolute constant $C$. Returning to \eqref{eq:Jensen}, we find that
    \begin{equation} \label{eq:low} \frac{1}{2\pi} \int_0^{2\pi} \log |f_n(Mre^{i\theta})| d\theta-\int_0^{Mr} \frac{n_-(t)}{t} dt \ge c\sqrt{nr}\end{equation}
for an absolute constant $c$ if $M$ is large enough.

We next apply Jensen's formula to $f_n(\varphi_+(z))$, where
    \[ \varphi_+(z):=M^2r^2 \frac{z+r+\alpha}{M^2r^2+(r+\alpha)z}, \]
where $0<\alpha<1$ is chosen so that $f_n(r+\alpha)\neq 0$. Let $n_+(t)$ be the counting function for the zeros $z_j$ of $f_n(\varphi_+(z))$ in the disc $|z|\le t$, where we only count zeros $z_j$ with $\re \varphi_+(z_j) >0$. Using crucially that the zeros of $f_n$ in the left half-plane contribute less in Jensen's formula for $f_n(\varphi_+(z))$ than they do in the formula for $f_n(\varphi_-(z))$, we obtain in a similar way as above that
    \[ \frac{1}{2\pi} \int_0^{2\pi} \log |f_n(Mre^{i\theta})| d\theta-\log|f_n(r+\alpha)|+O(M^{-1/2} \sqrt{rn} )-\int_0^{Mr} \frac{n_+(t)}{t} dt \le \int_0^{Mr} \frac{n_-(t)}{t} dt. \]
Taking into account \eqref{eq:low} and also \eqref{eq:outsidepw}, we therefore find that
    \begin{equation} \label{eq:almost} \int_0^{Mr} \frac{n_+(t)}{t} dt \ge c' \sqrt{rn} - \log |f_n(r+\alpha)| \end{equation}
for an absolute constant $c'$ if $M$ is large enough.

We will finally show that \eqref{eq:almost} along with another application of   Jensen's formula yields
    \begin{equation} \label{eq:final} \int_{mr}^{Mr} \frac{n_+(t)}{t} dt \ge c' \sqrt{rn} \end{equation}
for an absolute constant $0<m<1$. This will imply the desired bound $n_+(Mr)\gg \sqrt{rn}$. To prove \eqref{eq:final}, we let $n_{r+\alpha}(t)$ denote the number of zeros of $f_n$ in the disc $|z-r-\alpha|\le t$ and note that
    \begin{equation} \label{eq:localJ} \log |f_n(r+\alpha)|= \frac{1}{2\pi} \int_0^{2\pi} \log |f_n(r+\alpha+mre^{i\theta})| d\theta - \int_0^{mr} \frac{n_{r+\alpha}(t)}{t} dt. \end{equation}
A computation shows that $n_+(t)\le n_{r+\alpha}(t)$ for $t\le mr$ if $m$ small enough. Hence combining \eqref{eq:almost} with \eqref{eq:localJ}, we see that
    \[ \int_{mr}^{Mr} \frac{n_+(t)}{t} dt \ge c' \sqrt{rn} - \frac{1}{2\pi} \int_0^{2\pi} \log |f_n(r+\alpha+mre^{i\theta})| d\theta \]
when $m$ is small enough. Now using \eqref{eq:cartw} and estimating crudely, we find that
    \[ \int_0^{2\pi} \log |f_n(r+\alpha+mre^{i\theta})| d\theta \ll m \sqrt{rn}, \]
where the implicit constant is absolute. Hence we obtain \eqref{eq:final} if $m$ is small enough.

%\emph{Remark.} We may also consider the Poisson integral of $\log |h_n|$ itself in the upper half-plane. This representation reveals that the boundedness of $h_n(1+n+iy)$ forces $|%f_n(x)|$ to be small on a fairly large set, since we need some compensation for the growth of $f_n$ on the negative real axis. (The decay at $+\infty$ is not enough for this.) My guess %is that we have such smallness for relatively small positive $x$ in $(0,n)$.

\section{The basis functions do not give a Riesz basis for \texorpdfstring{$\mathcal H$}{H}}
\label{sec:RB}

This short section deals with a problem that arises naturally from the work of Kulikov, Nazarov, and Sodin \cite{KNS}. We recall from the introduction that $\mathcal{H}$ is defined as the collection of functions $f$ satisfying
\[ \| f\|_{\mathcal{H}}^2\coloneqq \int_{-\infty}^\infty \big(|f(x)|^2 +|\widehat{f}(x)|^2\big)(1+x^2) dx 
%+\int_{-\infty}^\infty |\widehat{f}(x)|^2 (1+x^2) dx
%<  \infty.
%\| f\|_{\mathcal{H}}^2\coloneqq \int_{-\infty}^\infty |f(x)|^2 (1+x^2) dx %+
%\int_{-\infty}^\infty |\widehat{f}(x)|^2 (1+x^2) dx
<  \infty.\]
This space was used in \cite{KNS}, where it was shown that all the Fourier uniqueness pairs $(\Lambda, W)$ constructed there yield frames in $\mathcal{H}$ (see \cite[Cor. 2]{KNS}. To be precise, this means that the reproducing kernels associated with $\Lambda$ and the Fourier transforms of the reproducing kernels associated with $W$, when normalized, constitute a frame in $\mathcal{H}$. In our case, it is natural to replace $\mathcal H$ by its subspace $\mathcal{H}_{\operatorname{e}}$ consisting of even functions in $\mathcal H$ and ask if the corresponding sequence in $\mathcal{H}_{\operatorname{e}}$ associated with the distinguished sequence $\{\sqrt{n}: n\ge 0\}$ enjoys the same property. We will now see that this is not the case. To this end, we let $K_x$ be the reproducing kernel of $\mathcal H_{\operatorname{e}}$ at $x$ so that
\[ f(x)=\langle f, K_x \rangle_{\mathcal H}\] 
for every $f$ in $\mathcal{H}_e$. The following is a corollary to Theorem~\ref{thm:moment}.
\begin{corollary}\label{cor:RB} The functions $a_0/\|a_0\|_{\mathcal H}$, 
$a_n/\|a_n\|_{\mathcal{H}}$, $\widehat{a}_n/\|\widehat{a}_n\|_{\mathcal{H}}$, $n=1,2,\ldots$, do not constitute a Riesz basis for $\mathcal{H}_{\operatorname{e}}$. 
Equivalently, the functions $(K_0+\widehat{K}_0)/\|K_0\|_{\mathcal H}$,
$K_{\sqrt{n}}/\|K_{\sqrt{n}}\|_{\mathcal H}$, $\widehat{K}_{\sqrt{n}}/\|K_{\sqrt{n}}\|_{\mathcal H}$, $n=1,2,\ldots$, do not constitute a Riesz basis for $\mathcal{H}_e$. \end{corollary}
We recall here that $(f_n)$ is a Riesz basis for a separable Hilbert space $H$ if $f_n=Te_n$ for some orthonormal basis $(e_n)$ for $H$ and some bounded invertible operator $T$ on $H$. The following fact is an immediate consequence of this definition. 
\begin{lemma} \label{lem:RB} Suppose that $( f_n )_{n=1}^\infty$ is a Riesz basis for a separable Hilbert space $H$. Then $\| f_n\|_{H}\asymp 1$. Moreover, there exists a unique sequence $(g_n)_{n=1}^\infty$ which is also a Riesz basis for $H$ and satisfies $\langle f_m,g_n \rangle_H=\delta_{m,n}$ for all $m$ and $n$. 
\end{lemma}
The sequence $K_0+\widehat{K}_0$, $K_{\sqrt{n}}, \widehat{K}_{\sqrt{n}}$, $n=1,2,\ldots$, is biorthogonal to the sequence of basis functions appearing in Corollary~\ref{cor:RB}. Hence Lemma~\ref{lem:RB} yields the equivalent formulation of the corollary in terms of reproducing kernels. To relate our result to that of \cite{KNS}, we note that if this biorthogonal sequence were a frame, it would have to be a Riesz basis, because we get an incomplete sequence on the removal of any one of its elements \cite[p. 90]{Gr}.
\begin{proof}[Proof of Corollary~\ref{cor:RB}] We begin by reiterating the point just made about biorthogonality. Indeed, in view of the interpolatary properties of the functions $a_n$ and $\widehat{a}_n$ (see Theorem A), we see that 
\[ \langle a_n, K_{\sqrt{m}} \rangle_{\mathcal H}=\delta_{m,n} 
\quad \text{and} \quad \langle \widehat{a}_n, K_{\sqrt{m}} \rangle_{\mathcal H}=0\]
when $m\ge 1$. Hence we have, up to normalization, identified a subsequence of the biorthogonal sequence whose existence is required by Lemma~\ref{lem:RB}.  We will now show that Lemma~\ref{lem:RB} is violated for this subsequence. 

We find that 
\begin{align} \nonumber  \| a_n\|^2_{\mathcal{H}} & 
\ge \frac14\int_{\sqrt{n}/2}^{\sqrt{n}} \big(|b_n^+(x)+b_n^-(x)|^2+|b_n^+(x)-b_n(x)|^2\big)(1+x^2) dx \\ & \ge \frac14 \int_{\sqrt{n}/2}^{\sqrt{n}} |b_n^+(x)|^2(1+x^2) dx \gg \sqrt{n} \log n , \label{eq:lowerbn} \end{align}
where we in the last step made a change of variable and used Lemma~\ref{prop:fnapprox} and Theorem~\ref{thm:moment}. 
On the other hand, the norm of point evaluation at $\sqrt{n}$ in $\mathcal{H}_{\operatorname{e}}$ is $\sim \sqrt{\pi/2}\,n^{-1/4}$, as follows from \cite[Thm. 5 and Lem. 7]{Z} and the fact that if $\tilde{K}_x$ is the reproducing kernel at $x$ of $\mathcal H$, then $K_x(y)=\tfrac12 (\tilde{K}_x(y)+\tilde{K}_{x}(-y))$. Thus if the vectors $a_n/\|a_n\|_{\mathcal{H}}$, $\widehat{a}_n/\|a_n\|_{\mathcal{H}}$ constituted a Riesz basis for $\mathcal{H}_{\operatorname{e}}$, then we would have
\begin{equation} \label{eq:pointev}  1=a_n(\sqrt{n}) \asymp \| a_n \|_{\mathcal{H}}\ n^{-1/4} \end{equation}
in view of Lemma~\ref{lem:RB}.
We now arrive at the desired conclusion, since \eqref{eq:pointev} is incompatible with \eqref{eq:lowerbn}.
\end{proof}

\section{Numerics and open questions}
\label{sec:numerics}

\subsection{Numerical calculation of \texorpdfstring{$f_n(x)$}{f\_n(x)}}
Since getting precise estimates for $f_n$ is difficult, it is interesting to get numerical data to obtain more insight into the behavior of these functions. One approach to computing $f_n(x)$ is to directly use the integral representation~\eqref{eq:fndef1}. This works relatively well for small values of~$n$ (roughly, for $n\le 100$), but becomes impractical for large $n$: the integrand is expensive to compute, highly oscillatory, and there is a lot of cancellation, so integration must be performed at very high precision.

An alternative approach is to compute the values $f_n(x)$ for a fixed non-negative value of~$x$ using their defining property, namely, the functional equation
    \[\sum_{n\ge0}f_n(x)e^{\pi i n\tau} + (\tau/i)^{-1/2}\sum_{n\ge0}f_n(x)e^{\pi i n(-1/\tau)} = e^{\pi i x \tau} + (\tau/i)^{-1/2}e^{\pi i x(-1/\tau)}\,.\]
The idea is to replace the above functional equation by a finite-dimensional system of linear equations on $f_n(x)$, $n\le N$, that one obtains by truncating the infinite sums and positing that the truncated equation holds for $\tau=\tau_0,\dots,\tau_N$ for suitably chosen $\tau_j$. Choosing $\tau_j$, $j=0,\dots,N$ to be equally spaced points on the line segment $[-1+\frac{10}{N}i,1+\frac{10}{N}i]$ (we chose the height of the line heuristically), we obtain a surprisingly numerically stable system of equations that can be solved even at machine precision and for very large values of~$N$. 

This method follows the procedure described by Hejhal in~\cite{Hejhal1992} to compute Maass cusp forms. We remark that more sophisticated versions of this method have been developed (see \cite{Hejhal1999,stroemberg,BMR}). These methods however assume that the imaginary part of the points along the line segment is smaller than the imaginary part of the points transformed via the functional equation and ensure that the linear system is diagonally dominant. This assumption does not hold in our setting due to the existence of the second cusp which has non-periodic boundary condition and therefore does not offer an expansion in powers of~$e^{\pi i \tau}$.

Although we cannot provide any rigorous error bounds for the values obtained using this method, for small~$n$ they are in perfect agreement with rigorous high-precision numerics for $f_n(x)$ done using~\eqref{eq:fndef1}. We additionally varied the height of the line segment to obtain heuristic estimates of the precision of the coefficients.

In Figure~\ref{fig:f500a} from the introduction we show a part of the plot of $f_{500}(x)$ obtained using the above method. In particular, one can observe the regular behavior of $f_n(x)/\sin(\pi x)$ for large $x$, in line with~\eqref{eq:fn9}. For smaller $x$ the values of $f_n(x)$ tend to be more chaotic, and~\eqref{eq:fn9} no longer gives a good approximation, but despite this the values remain small for all~$n$. In Figure~\ref{fig:bulk} we show simultaneous plots of $f_n(x)$ for $0\le x\le 2$ and $100\le n\le 150$, revealing that on the interval $[0,2]$ the values of this collection of functions never exceed~$1$ in absolute value. 

\begin{figure}[ht]
    \includegraphics[width=16cm,height=6.5cm]{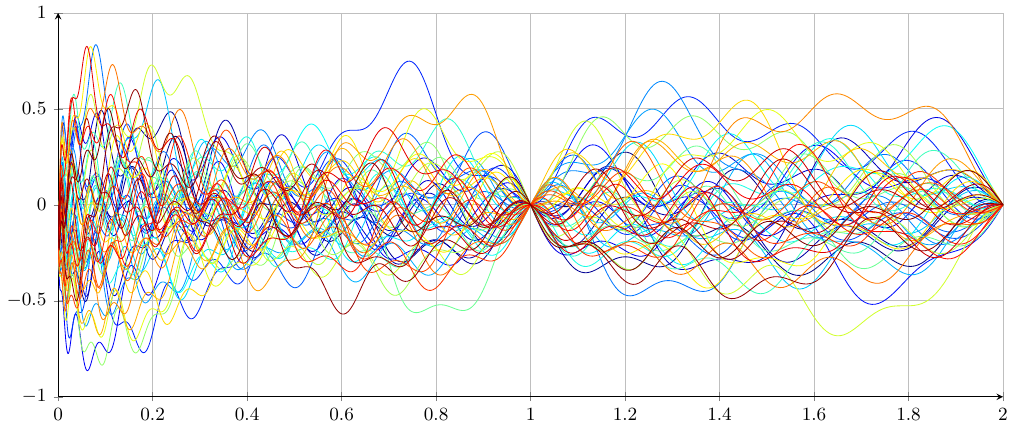}
    \caption{Plots of $f_{n}$ for $100\le n \le 150$.}
    \label{fig:bulk}
\end{figure}

\subsection{Questions and conjectures} 
Based on a numerical computation of the first 10000 basis functions, we have found very few instances of $f_n$ at positive arguments taking values larger than 1 in absolute value and no instances where the value was larger than $2$. It is therefore tempting to conjecture the following.
\begin{conjecture}
    There exists $C>0$ such that $|f_n(x)|<C$ holds for all $n\in\Z_{\ge0}$, $x\ge 0$.
\end{conjecture}

Computations of $f_n(x)$ for even larger $n$ reveal a more curious behavior.
In Figure~\ref{fig:histogram} we show the distribution of values $f_{n}(x)n^{1/4}$ for $0\le n\le 50000$. This picture closely resembles that of a normal distribution, and at the very least seems to indicate that the values $f_{n}(x)n^{1/4}$, $n\ge1$ do indeed have a limiting distribution. We do not know if this limiting distribution can be normal (computing empirical moments points to this not being the case), but the following question seems reasonable.
\begin{question}
    Does the sequence $\{n^{1/4}f_n(x)\}_{n\ge0}$ have a statistical distribution with finite moments? If not, is this true for $\{n^{1/4}\log^{\alpha}(n)f_n(x)\}_{n\ge2}$ for some $\alpha\in\R$?
\end{question}

\begin{figure}[ht]
    \includegraphics[width=10cm,height=6cm]{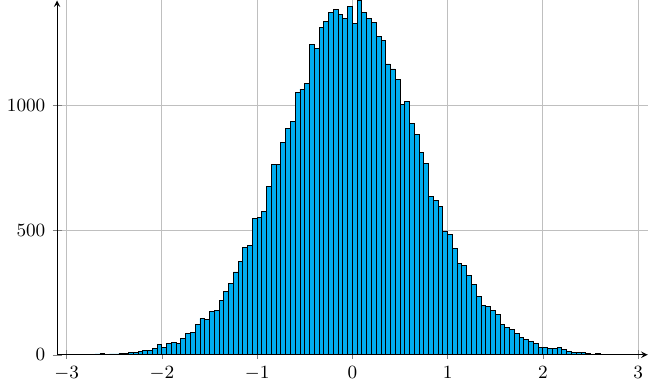}
    \caption{Histogram of $n^{1/4}f_{n}(0.63)$, $0\le n \le 50000$.}
    \label{fig:histogram}
\end{figure}

If the answer is positive, then for a fixed value of $x$, we expect $f_n(x)$ to be on average of size $n^{-1/4}$ (or possibly $n^{-1/4}(\log n)^{\alpha}$ for some $\alpha$). Thus, roughly speaking, we expect the Fourier coefficients of the generating series $F(\tau,x)$ to satisfy an analogue of the Ramanujan--Petersson conjecture. Note that the first estimate in~\eqref{eq:pw} gives an even stronger bound $f_n(x)=O(n^{-0.392})$ for $n/8 < x < n-\sqrt{n}$. In the following conjecture, we propose a further improvement  in this regime.
\begin{conjecture}
For any $\eps,\delta >0$ we have $f_n(x) = O_{\eps,\delta}(n^{-1/2+\delta})$ for $\eps n \le x \le n-\sqrt{n}$.
\end{conjecture}
For $x\in[n/8,n-\sqrt{n}]$ this conjecture is equivalent to $\Psi(t)$ defined in~\eqref{eq:psidef} growing more slowly than any positive power of $t$. As can be seen from the proof of Theorem~\ref{thm:lindelof}, this would follow from the exponent pair conjecture. We conjecture an even stronger property.
\begin{conjecture}
There exists $k>0$ such that $\Psi(x)=\sum_{n\ge1}\tfrac{2}{\sqrt{n}}\cos(\pi(\tfrac{3n-1}{4}-\tfrac{x}{n}))$ satisfies 
    \[\Psi(x)=O(\log^k|x|)\,,\qquad x\to\pm\infty\,.\]
\end{conjecture}
The limited numerical evidence that we have is compatible with the validity of the above conjecture for $k=2+\eps$.

Next, motivated by Theorem~\ref{thm:sqrsum}, in Figure~\ref{fig:l2norms} we plot the numerical values of $\int_{0}^{\infty}|f_n(x)|^2dx$ for $1\le n\le 1000$ in comparison with a logarithmic curve. The similarity between the plots suggests the following.

\begin{figure}[ht]
    \includegraphics[width=13cm,height=8.2cm]{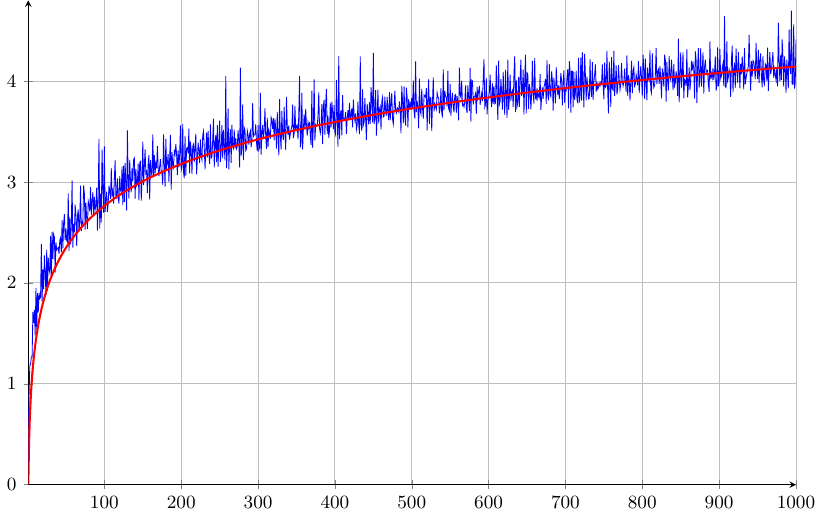}
    \caption{Values $\int_{0}^{\infty}|f_n(x)|^2dx$, compared with $0.6\log n$ for $1\le n\le 1000$.}
    \label{fig:l2norms}
\end{figure}

\begin{conjecture} For $n>1$, we have
    \begin{equation} \label{eq:sqrsumconj}  
    \int_0^\infty |f_n(x)|^2 dx \asymp \log n\,.
    \end{equation}    
\end{conjecture}
Note that the lower bound $\int_0^\infty |f_n(x)|^2 dx \gg \log n$ is valid because of~\eqref{eq:fnsecmom}, so the only question is about the upper bound. If this conjecture is true, then the lower bound in Theorem~\ref{thm:sqrsum} is optimal up to a multiplicative constant.
    
We saw in the previous section that the normalized basis functions do not constitute a Riesz basis for $\mathcal H_{e}$. The following conjecture suggests that the sequence $\Lambda_0\coloneqq \{\sqrt{n}: \ n\ge 0\}$ should enjoy  a considerably weaker property.
\begin{conjecture} 
$(\Lambda_0, \Lambda_0)$ is a Fourier uniqueness pair for $\mathcal{H}_{\operatorname{e}}$.\end{conjecture}
A related question is whether the normalized basis functions could give a Riesz basis for a Fourier invariant Hilbert space that is somewhat smaller than ${\mathcal H}_e$.

\end{document}